\documentclass[12pt,reqno]{amsart}
  
\usepackage{latexsym}
\usepackage{amssymb, tikz}

\usepackage{amsthm}

\newcommand{\defeq}{\stackrel{{\rm def}}{=}}

\newcommand{\ra}{\rangle}
\newcommand{\la}{\langle}

\newcommand{\eps}{\epsilon}
\newcommand{\llambda}{\lambda}

\newcommand{\vareps}{\varepsilon}

\newcommand{\Spec}{{\operatorname{Spec}}}

\newcommand{\IN}{{\mathbb N}}
\newcommand{\R}{{\mathbb R}}
\newcommand{\IR}{{\mathbb R}}
\newcommand{\C}{{\mathbb C}}

\newcommand{\IT}{{\mathbb T}}
\newcommand{\Z}{{\mathbb Z}}
\newcommand{\IZ}{{\mathbb Z}}

\def\bbbone{{\mathchoice {1\mskip-4mu {\rm{l}}} {1\mskip-4mu {\rm{l}}}
{ 1\mskip-4.5mu {\rm{l}}} { 1\mskip-5mu {\rm{l}}}}}

\renewcommand{\tt}{\tilde{t}}
\newcommand{\tP}{\tilde{P}}
\newcommand{\tp}{\tilde{p}}
\newcommand{\tPsi}{\tilde{\Psi}}
\newcommand{\tPhi}{\tilde{\Phi}}
\newcommand{\tpsi}{\tilde{\psi}}
\newcommand{\tvarphi}{\tilde{\varphi}}
\newcommand{\tTheta}{\tilde{\Theta}}
\newcommand{\h}{\hbar}
\newcommand{\Oph}{\operatorname{Op}_{\h}}

\renewcommand{\d}{\partial}

\newcommand{\neigh}{{\operatorname{neigh}}}
\renewcommand{\phi}{\varphi}

\newcommand{\cD}{\mathcal{D}}

\newcommand{\cF}{\mathcal{F}}

\newcommand{\cO}{\mathcal{O}}

\newcommand{\cS}{\mathcal{S}}

\newcommand{\cU}{\mathcal{U}}

\newcommand{\bequ}{\begin{equation}}
\newcommand{\eequ}{\end{equation}}

\theoremstyle{plain}

\newtheorem*{theorem*}{Theorem}

\numberwithin{equation}{section}

\newtheorem{theo}[equation]{{\sc Theorem}}

\newtheorem{cor}[equation]{{\sc Corollary}}

\newtheorem{maindefn}[equation]{{\sc Definition}}
\newtheorem{lem}[equation]{{\sc Lemma}}

\newtheorem{prop}[equation]{{\sc Proposition}}

\theoremstyle{definition}

\theoremstyle{remark}
\newtheorem{rem}[equation]{Remark}

\theoremstyle{note}

\title[Logarithmic quasimodes]{Strong scarring of logarithmic quasimodes}
\author{Suresh Eswarathasan}
\address{Department of Mathematics and Statistics, McGill University, 805 Rue Sherbrooke Ouest, Montr\'eal, Canada}
\email{suresh@math.mcgill.ca}
\author{St\'ephane Nonnenmacher}
\address{Institut de Physique Th\'eorique, Universit\'e Paris-Saclay,
  Commissariat \`a l'\'energie atomique, 91191 Gif-sur-Yvette, France}
\email{stephane.nonnenmacher@cea.fr}

\date{}

\begin{document}

\begin{abstract}
We consider a semiclassical (pseudo)differential 
operator on a compact
surface $(M,g)$, such that the Hamiltonian flow generated by its principal symbol 
admits a hyperbolic periodic orbit $\gamma$ at some energy $E_0$. 
For any $\vareps>0$, we then explicitly construct families of {\it quasimodes} of this operator, satisfying 
an energy width of order $\vareps\frac{\h}{|\log\h|}$ in the semiclassical limit, but which still exhibit
a ``strong scar" on the orbit $\gamma$, {\it i.e.} that these states have a positive weight in any microlocal neighbourhood of $\gamma$. We pay attention to optimizing the constants involved in the estimates.
This result generalizes a recent result of Brooks \cite{Br13} in the case of hyperbolic surfaces.
Our construction, inspired by the works of Vergini {\it et al.} in the physics literature, relies on controlling the propagation of Gaussian wavepackets up to the Ehrenfest time.
\end{abstract}

\maketitle

\section{Introduction}

The main objective of Quantum Chaos is to understand the behavior of wave or quantum propagation when the underlying classical Hamiltonian dynamics is chaotic, or at least presents various forms of hyperbolicity. Semiclassical analysis shows that the wave mechanical system should ``converge towards" its classical counterpart, in the high-frequency regime; the latter is equivalent with the semiclassical regime $\hbar \to 0$, where $\hbar$ is an effective Planck's constant. 
However, when the classical dynamics presents some hyperbolicity, this {\it semiclassical correspondence} breaks down when the time of evolution exceeds the so-called Ehrenfest time $T_E\sim c_0 |\log h|$, where $c_0>0$ depends on the classical dynamics. As a result, a precise description of the stationary states (eigenstates) of the quantum system remains a challenge, since a precise description usually requires one to understand the quantum evolution up to much longer times.

Although our results apply to more general situations, we will in this introduction focus on the case of the Laplace-Beltrami $\Delta_g$ on a compact Riemannian manifold $(M,g)$.   
To describe the high frequency dynamics generated by this operator, we will use a semiclassical formalism, that is define the {\it semiclassically rescaled Laplacian} 
$$
P(\h)=-\h^2 \Delta _g\,,
$$ 
and consider its  spectrum in the vicinity of a fixed value $E=1$, in the semiclassical regime $\h\ll 1$. The operator  $P(\h)$ generates the dynamics of a free quantum particle on $M$, embodied in the Schr\"odinger equation
\begin{equation}\label{e:Schro0}
i\h \frac{\d }{\d t}\varphi (t) = P(\h) \varphi (t)\,,
\end{equation}
where $\varphi(t)\in L^2(M,dg)$ is the wavefunction of the particle at time $t$. In the semiclassical limit, this dynamics can be related with the free motion of a classical particle on $M$, namely the geodesic flow. This flow, which lives on the phase space $T^*M\ni (x,\xi)$, is generated by the Hamiltonian function $p(x,\xi)=\|\xi \|^2_g$, so we will denote it by $\Phi^t_p$.
By scaling, we may restrict this flow to the unit cotangent bundle $S^*M=p^{-1}(1)$.

When $M$ is compact without boundary and the metric $g$ has everywhere negative sectional curvature, the geodesic flow is of Anosov type: every trajectory is hyperbolic. Ergodic theory shows that this flow is ergodic and mixing with respect to the Liouville measure on $S^*M$. Anosov flows have been thoroughly studied, and represent the strongest form of chaos. One objective of Quantum Chaos is to study the {\it localization properties} of the {\it eigenstates} of $P(\h)$ at energies $E_\h\sim 1$, whilst taking into account the chaotic properties of the geodesic flow. The semiclassical r\'egime $\h\to 0$ amounts to studying the eigenmodes $\varphi_j$ of the Laplacian of frequencies $\lambda_j \sim \h^{-1}\to\infty$. 

The Quantum Ergodicity Theorem \cite{Sch74,CdV85,Zel87} shows that if the geodesic flow is ergodic, then {\it almost all} the high-frequency eigenstates $\varphi_j$ are asymptotically equidistributed over $M$. Namely, there exists a subsequence $\cS\subset \IN$ of density one, such that the probability measures $|\varphi_j(x)|^2\,dx$ weakly converge to the normalized Lebesgue measure when $\cS\ni j\to\infty$.  An analogue  of this theorem is established in the boundary setting for Dirichlet eigenstates by G\'erard-Leichtnam \cite{GL93} for domains in $\mathbb{R}^n$ with $W^{2, \infty}$ regularity and Zelditch-Zworski \cite{ZZ96} for compact manifolds with piecewise smooth boundaries.  

An open question concerns the existence of {\it exceptional} eigenstates, which would localize in a nonuniform way on $M$. Rudnick and Sarnak conjectured that for $(M,g)$ of negative curvature, such exceptional eigenstates do not exist \cite{RS94}, a property called Quantum Unique Ergodicity. So far this conjecture was proved only in the case of certain hyperbolic surfaces with arithmetic properties \cite{Lin06,BL14}.
Numerical computations of eigenstates of ergodic Euclidean billiards \cite{H84} have shown that certain high frequency eigenstates present an enhanced intensity along short closed geodesics. In the case these orbits are hyperbolic, a quantitative characterization of these enhancements --- baptized as {\it scars} of the geodesic $\gamma$ on the corresponding eigenstates --- has remained elusive, in spite of intensive investigations in the physics literature (see e.g. the review \cite{Kap99} and references therein), or a more recent numerical study by Barnett \cite{Bar06}). 
In particular, it is unclear whether these enhancements can take the form of singular components $w_\gamma\delta_\gamma$, $w_\gamma>0$, in the weak limits of the measures $|\varphi_j(x)|^2\,dx$.  

To connect the localization of eigenstates with the classical dynamics, which takes place in the phase space $T^*M$, it is convenient to lift the localization properties to phase space, that is characterize the {\it semiclassical microlocalization} of the eigenstates. More generally, for any semiclassical sequence of $L^2$-normalized states $(\varphi_\h)_{\h\to 0}$, we may characterize the asymptotic phase space localization of these states through the {\it semiclassical measures} associated with this sequence. Let us recall how these measures are constructed.
The sequence of normalized states $(\phi_\h)_{\h\to 0}$ is said to converge towards the semiclassical measure $\mu_{sc}$ if, for any observable $a\in C^\infty_c(T^*M)$, one as
$$
\la \phi_\h,\Oph(a)\phi_\h\ra\xrightarrow{\h\to 0}\int_{T^*M} a\,d\mu_{sc}.
$$
Here $\Oph(a)$ is an operator on $L^2(M)$, obtained through a semiclassical quantization of the observable $a(x,\xi)$ (see section~\ref{hprelim}). The limit measure is independent of the choice of quantization. 
For any sequence $(\phi_\h)$ one can always extract a subsequence $(\phi_{\h_j})$ converging to some semiclassical measure $\mu_{sc}$. The latter is then said to be {\it a} semiclassical measure associated with the sequence $(\phi_\h)$. Semiclassical measures provide a notion of phase space localization, or {\it microlocalization}: a sequence $(\phi_\h)$ will be said to be microlocalized in a set $K\subset T^*M$ if any associated measure is supported inside $K$.

Coming back to the scarring phenomenon, a sequence of eigenstates $(\varphi_\h)$ of $P(\h)$, of energies $E_\h\approx 1$, is said to exhibit a {\it strong scar} on a closed orbit $\gamma\in S^*X$ if any associated semiclassical measure contains a component $w_\gamma\delta_\gamma$ with some positive weight $w_\gamma>0$ (here $\delta_\gamma$ is the normalized, flow invariant Dirac measure on the orbit $\gamma$). Hassell has shown \cite{Hass10} that such strong scars may exist in the case the orbit $\gamma$ is marginally stable (and belongs to a family of ``bouncing ball orbits''). 
In negative curvature, the existence of such exceptional eigenstates would contradict the QUE conjecture. In \cite{An06,AN07} it was shown that the eigenstates cannot fully localize along $\gamma$: for any sequence of eigenstates, the weight $w_\gamma$ is necessarily smaller than unity, in particular in the case of {\it constant} negative curvature one must have $w_\gamma\leq 1/2$. So far exceptional eigenstates have been exhibited only 
for a toy model of Anosov system, namely the hyperbolic automorphisms of the 2-dimensional torus, casually called ``Arnold's cat map" \cite{FND03}. The classical dynamics arises from a hyperbolic symplectomorphism on $\IT^2$, which is represented by a matrix $A\in SL(2,\Z)$; it can be quantized into a family of unitary operators $(U_\h)_{\h\to 0}$ of finite dimensions $\sim \h^{-1}$. One can analyze the microlocalization on $\IT^2$ of the eigenstates of the operators $U_\h$. For any periodic orbit $\gamma$ of the classical map, the very special algebraic properties of the operators $U_\h$ allows to explicitly construct eigenstates with strong scars on $\gamma$, with weights $w_\gamma$ necessarily taking values in $[0,1/2]$ \cite{FN04}. These special eigenstates were constructed by linear combinations of evolved Gaussian wavepackets localized on $\gamma$. 

In the present article we will consider a similar construction for the dynamics provided by the geodesic flow on $(M,g)$, or more general Hamiltonian flows.  
By doing so we won't be able to construct eigenstates of $P(\h)$, but only approximate eigenstates, which are called {\it quasimodes} of $P(\h)$. Let us now give a precise definition of this notion.

\begin{maindefn}\label{def:qmodes}
For a given energy level $E>0$, we say that a semiclassical family $(\phi_{\h})_{\h\to 0}$ of $L^2$ normalized states is a family of \emph{quasimodes} of $P_\h$ of central energy $E$ and \emph{width} $f(\h)$, if and only if there exists $\h_0>0$ such that \begin{equation}\label{e:QM1}
\| (P(\h) - E) \phi_{\h} \|_{L^2(M)}  \leq  f(\h)\,,\qquad \forall \h\in (0,\h_0]\,.
\end{equation}
A slightly broader definition consists in allowing the central energy to depend on $\h$ as well, namely considering a sequence $(E_\h)_\h$ with $E_\h \to E$, and requiring 
\begin{equation}\label{e:QM2}
\| (P(\h) - E_\h) \phi_{\h} \|_{L^2(M)}  \leq  f(\h)\,,\qquad \forall\h\in (0,\h_0]\,.
\end{equation}
\end{maindefn}
\begin{rem}
One way to construct quasimodes of center $E_\h$ and width $f(\h)$ is to take linear combinations of eigenfunctions of $P(\h)$ with eigenvalues in the interval $ [E_\h - f(\h), E_\h + f(\h)]$. We will use this trick in the last stage of the proof of Theorem~\ref{thm:mainthm} below. However, the construction of a fully localized quasimode in Prop.~\ref{pro:quasim0} will proceed differently, namely by a time averaging procedure. 
\end{rem}
What is the interest of considering quasimodes instead of eigenstates? Except for very special systems like the quantized ``cat map'', we are unable to give explicit, or even approximate expression of the eigenstates of Anosov systems. On the opposite, we will construct explicit quasimodes of certain widths $f(\h)$. The larger the width, the less constrained the localization properties. Using the fact that the function $p(x,\xi)$ is the principal symbol of the operator $P(\h)$, one can show the following properties of semiclassical measures associated with $(\phi_\h)$:
\begin{itemize}
\item if $f(\h)=o(1)_{\h\to 0}$, then $\mu_{sc}$ is supported on $p^{-1}(E)$, and is a probability measure.
\item if $f(\h)=o(\h)$, then $\mu_{sc}$ must be invariant through the Hamiltonian flow $\Phi^t_p$. 
\end{itemize}
For an Anosov flow like the geodesic flow on $(M,g)$, there exist many invariant measures, the simplest ones being the Liouville measure on $p^{-1}(E)$, or the delta measures localized along a closed orbit $\gamma\subset p^{-1}(E)$. 
Our Prop.~\ref{pro:quasim0} will exhibit quasimodes of width $C\h/|\log\h|$ (with $C>0$ large enough) converging to the measure $\delta_\gamma$. An easy corollary is that, if the flow $\Phi^t_p$ is Anosov on $p^{-1}(E)$, then for any flow-invariant measure $\mu$ on $p^{-1}(E)$ one may construct a sequence of quasimodes of width $C\h/|\log\h|$ converging to $\mu$.

In \cite{An06,AN07} it was specified that the semiclassical measures associated with quasimodes of width $f(\h)=o\big(\frac{\h}{|\log\h|} \big)$ must satisfy the same ``half-delocalization'' constraints as the measures associated with eigenstates. Hence, for such Anosov flows the ``logarithmic scale'' $c\frac{\h}{|\log\h|}$ seems to be critical for the quasimode width, the constraints on the localization depending on the size of the factor $c>0$. In \cite{Br13} Shimon Brooks studied such ``logarithmic quasimodes'' for the Laplacian on compact hyperbolic surfaces. He showed that quasimodes of width $c\frac{\h}{|\log\h|}$ with $c>0$ arbitrary small could still be ``strongly scarred'' on $\gamma$:
\begin{theo}  \cite{Br13}\label{thm:Brooks13}
Let $(M,g)$ be a compact hyperbolic surface, and $\gamma \subset S^*M$ a closed geodesic.  Then for any $\epsilon > 0$, there exists $\delta(\epsilon) >0$ and a sequence $(\varphi_\h)_{\h\to 0}$ of quasimodes of center energy $E=1$ and width $\epsilon \frac{\h}{|\log \h|}$, converging to a semiclassical measure $\mu_{sc}$ satisfying $\mu_{sc}(\gamma)\geq   \delta(\epsilon)$.
\end{theo}
The methods used in \cite{Br13} are quite specific to the setting of compact hyperbolic surfaces.  Our main objective in this article is to generalize Brooks's theorem to more general surfaces and Hamiltonian flows. Our construction will be mostly local, independent of the global dynamics on $p^{-1}(E)$; essentially, the only assumption we need to make on the dynamics is the presence of a hyperbolic closed orbit $\gamma$. 
We will also attempt to optimize the relation between the width constant $\vareps>0$ and the weight $\mu_{sc}(\{\gamma\})$. In principle our methods could be extended to higher dimensions, but probably at the expense of this optimization (see the Remark~\ref{r:higher-dims} at the end of section~\ref{s:QNF}). 
Our main result is the following Theorem.
\begin{theo} \label{thm:mainthm}
Let $(M,g)$ be a smooth compact Riemannian surface without boundary. Let $P(\h)\in \Psi^m_\h(M)$ be a self-adjoint elliptic pseudodifferential operator, with semiclassical principal symbol $p(x,\xi)$. For a regular energy level $E_0$, assume that the Hamiltonian flow $\Phi^t_p$ admits a closed hyperbolic orbit $\gamma \subset p^{-1}(E_0)$.

Choose any $\epsilon>0$.  
Then there exists $\delta(\epsilon, \gamma)>0$, a sequence of energies $(E_\h = E_0 + \mathcal{O}(\h / | \log \h | )$ and a sequence of quasimodes $(\psi_{\h})_{\h\to 0}$ centered on $E_\h$ and of width $\epsilon ( \h / | \log \h |)$, such that any semiclassical measure $\mu_{sc}$ associated with $(\psi_{\h})_{\h\to 0}$ satisfies
$$
\mu_{sc}(\{\gamma\})\geq \delta(\epsilon, \gamma)\,.
$$
Furthermore, we have an explicit estimate for $\delta(\epsilon, \gamma)$ in the r\'egime $\eps\ll 1$:
\bequ\label{e:weight}
\delta(\epsilon,\gamma) = \frac{\eps}{\pi\lambda_\gamma}\frac{2}{3\sqrt{3}} + \cO((\eps/\lambda_\gamma)^2)\,,
\eequ
where $\lambda_\gamma>0$ is the expanding rate per unit time along the unstable direction of the orbit $\gamma$.
\end{theo}
In the course of proving this theorem we will first construct logarithmic quasimodes fully localized along $\gamma$, that is converging to the semiclassical measure $\delta_\gamma$, but with a lower bound on their width.
\begin{prop}\label{pro:quasim0}
Let $(M,g)$, $P(\h)$ and $\gamma\subset p^{-1}(E_0)$ be as in the above theorem. 

Then, for any sequence of energies $(E_\h= E_0+\cO(\h))_{\h\to 0}$ and any constant $C_\gamma>\pi\lambda_\gamma$, there exist a family of quasimodes $(\psi_{\h}\in L^2(M))_{\h\to 0}$ centered on $E_\h$ and of width $ C_\gamma ( \h / | \log \h |)$,
converging to the semiclassical measure $\mu_{sc}=\delta_\gamma$. 
\end{prop}
\begin{rem}
Notice a difference between the ``central energies'' $E_\h$ in the Theorem and the Proposition: in the latter we are free to choose the energy $E_\h$ in the range $E_0+\cO(\h)$, while in the theorem the energy $E_\h$ is given to us. 
\end{rem}

The dynamics covered by our result include various types of dynamics: the Anosov systems described above, for which it seems for the moment very difficult to explicitly describe quasimodes of widths $o(\h/|\log\h|)$.  On the other hand, our result also encompasses 2-dimensional Liouville-integrable Hamiltonian flows featuring one hyperbolic orbit: in this case one can generally obtain a precise description of quasimodes of width $\cO(\h^\infty)$, or even of the eigenstates of $P(\h)$. We will discuss this case a bit in the next subsection.

\subsection{Discussion of related results}

The construction of quasimodes which localize along closed geodesics, or more generally along invariant submanifolds, has a long history, starting at least with the work of Keller \cite{K58}, who initiated the construction of quasimodes supported on invariant Lagrangian submanifolds of $T^*M$, using WKB-type Ans\"atze.  The methods were further developed during the '70s and '80s, at least in the case of completely integrable dynamics, by a plethora of authors including Maslov \cite{Mas72}, Duistermaat \cite{Duis74}, Weinstein \cite{Wein74}, and Colin de Verdi\`ere \cite{CdV77}, through the use of ``softer" symplectic geometric methods and Fourier integral operators. One could then exhibit quasimodes with widths of order $O(\h^\infty)$, localized on some invariant Liouville-Arnold torus. 

Babich and Lazutkin \cite{BabLaz68} constructed $\cO(\h^\infty)$ quasimodes concentrating along a single closed geodesic {\it of elliptic type} (that is, the linearized Poincar\'e map of the orbit $\gamma$ is an elliptic matrix).  The works of Voros \cite{Vor76}, Guillemin-Weinstein \cite{GuiWein76}, and Ralston \cite{Ral76} each carried out this idea via different methods; in particular Ralston constructed maximally localized eigenstates in the form of {\it Gaussian beams}, which are minimal uncertainty Gaussian wavepackets transversally to the orbit. These constructions were specific to the case of elliptic orbits, albeit with nonresonance conditions. 

In the case of a hyperbolic orbit, the construction of $\cO(\h^\infty)$ quasimodes using Gaussian wavepackets usually breaks down, due to the fast spreading of the wavepacket along the unstable direction. However, for certain Liouville-integrable systems, the particular structure of the unstable and stable manifolds of $\gamma$ allows to construct $\cO(\h^\infty)$ quasimodes as WKB states along the various branches of these manifolds. This construction, peformed by Colin de Verdi\`ere and Parisse \cite{CdVP94} in the case of a surface of revolution, shows that some family of eigenstates of the Laplacian converge to the semiclassical measure $\delta_\gamma$, but the convergence is rather slow: the weight of the measure $|\varphi_\h(x)|^2dx$ on the complement of a small neighbourhood of $\gamma$ decays at the rate $|\log\h|^{-1}$. We will come back to the Colin de Verdi\`ere-Parisse quasimode construction in subsection~\ref{s:CdVP}.
Toth studied specific integrable systems in higher dimension \cite{To96,To99}, and showed that some families of eigenstates may converge to $\delta_\gamma$ for $\gamma$ a hyperbolic orbit.
The inverse logarithmic decay of the measure away from $\gamma$ was shown to hold in greater generality by Burq-Zworski \cite{BZ04}: their Theorem 2' implies that in any dimension, a quasimode of width $c\h/|\log\h|$ with $c>0$ small enough must have a weight $\geq C^{-1}|\log\h|^{-1}$ outside a small neighbourhood of a hyperbolic closed orbit $\gamma$. These results were generalized to semi-hyperbolic orbits by Christianson \cite{Chr07,Chr11}, including the case of a manifold with boundary. 

In the Quantum Chaos physics literature several studies were devoted to constructing various forms of localized states, which could hopefully mimick the ``scars" numerically observed on certain eigenfunctions, or at least exhibit significant overlaps with these eigenstates. These localized states are generally based on Gaussian wavepackets localized on a point of $\gamma$. De Polavieja {\it et al.} \cite{dPBB94}, and then Kaplan and Heller \cite{KH99} used time averaging of such a wavepacket to construct a localized quasimode, but neither sets of authors tried to optimize the time of evolution or the width. On the opposite, a series of works by Vergini and collaborators  constructed so-called ``scar functions" localized along $\gamma$. They first considered a Gaussian beam along $\gamma$ \cite{VC00}, then used several procedures to improve the width of the quasimode, keeping the latter localized near $\gamma$ \cite{VC01,VS05}. In particular, in \cite{VS05} they improved the Gaussian beam through a time averaging procedure, and  optimized the width by evolving up to the Ehrenfest time and by carefully selecting the weight function.

In the mathematics literature, Brooks \cite{Br13} also constructs localized logarithmic quasimodes in the course of proving his Thm~\ref{thm:Brooks13}; his method also relies on time averaging, but starts from ``radial states" specific to the geometry of hyperbolic surfaces, instead of Gaussian wavepackets. As Brooks did not try to optimize his constants, one of our aims in this article is to make those precise.

\subsection{Outline of the proof}\label{s:outline}

Most of the article deals with the proof of Prop.~\ref{pro:quasim0}, namely
the construction of the fully localized quasimodes. Our proof is essentially identical to the procedure used by Vergini-Schneider in \cite{VS05}. The main difference is that the semiclassical tools we use ({\it e.g.} the Quantum Normal Form) are mathematically rigorous, including our remainder estimates. The extension to the partially localized quasimodes of Thm~\ref{thm:mainthm} is similar in spirit to that of \cite{Br13}, albeit with a better control of the constants.

We now present an outline of the proof of Theorem \ref{thm:mainthm}, which starts by constructing the fully localized logarithmic quasimodes of Prop.~\ref{pro:quasim0}.

1) \textbf{Quantum Normal Form} The first task is to choose convenient canonical coordinates near the orbit $\gamma$, namely {\it normal form} coordinates, which map the neighbourhood of $\gamma=\gamma(E_0)\subset p^{-1}(E_0)$ into a neighbourhood of the circle  $\{\tau=x=\xi=0,\,s\in\IT \}$ inside the model phase space $T^*(\IT_s\times \IR_x)$ spanned by the coordinates $(s,x;\tau,\xi)$. In these coordinates, the Hamiltonian $p$ takes the following form, at the order $2N$ in the transverse coordinates:
$$
\tp(s,x;\tau,\xi)= f(\tau) + \llambda(\tau)x\xi + \sum_{2\leq \alpha\leq N}q_\alpha(\tau)(x\xi)^\alpha+\cO((x,\xi)^{2N+1})\,,
$$
where $f(\tau)=E_0+\frac{\tau}{T_0}+\cO(\tau^2)$, with $T_0$ the period of $\gamma=\gamma(E_0)$ and $\llambda(0)>0$ the expansion rate along the unstable direction. 
This normal form can be extended to the quantum level, following works of G\'erard-Sj\"ostrand and Sj\"ostrand 
\cite{GS87,Sj02}. The result is that,  {\it microlocally near $\gamma$}, the operator $P(\h)$ is unitarily equivalent to a Quantum Normal Form (QNF) operator $\tP(\h)$, a pseudodifferential operator on $\IR\times \IR$ of ``simple'' form; in particular, its principal symbol of $\tP(\h)$ is given by $\tilde{p}$ above.

In this outline we will rather consider a simplified normal form to focus on the important ideas. Namely, we will take the following quadratic Hamiltonian on the phase space $T^*(\IT\times \IR)$:
$$
p_2(s,x;\tau,\xi)\defeq E_0+\frac{\tau}{T_0} + \llambda\,x\,\xi\,,,
$$ 
with $\llambda=\llambda(0)$. The stable (resp. unstable) subspace corresponds to the $\xi$-axis (resp. $x$-axis). 
The corresponding quantum operator is obtained through the Weyl quantization of $p_2$ on $\IT\times \IR\ni (s,x)$:
$$
\tP_2(\h) =\Oph^w(p_2)= E_0 + \frac{1}{T_0} \h D_s + \llambda \big(x\,\h D_x - i\h/2 \big),\qquad s\in \IT=\IR/\IZ,\ x\in \IR\,.
$$
On the right hand side we recognize the selfadjoint 1D dilation operator $\big(x\,\h D_x - i\h/2 \big)$.

2) \textbf{Propagation of transverse Gaussians}:  The state with which we start has the shape of a Gaussian beam: it is a minimal uncertainty Gaussian wavepacket transversely to $\gamma$, and a plane wave along the direction of $\gamma$. In our normal form coordinates $(s,x)$, this state has the form
\begin{equation}\label{e:psi0}
\psi_0(s,x) = \frac{1}{(\pi \h)^{1/4}}\, \exp(2 i \pi  n s)\, \exp\big(-\frac{x^2}{2\h}\big),\qquad \text{for some choice of }n\in\Z.
\end{equation}
This state has, with respect to the model Hamiltonian $\tP_2(\h)$, an average energy 
$E_\h\la \psi_0,\tP_2(\h)\psi_0\ra = \la \psi_0,\tP_2\psi_0\ra = E_0 + \frac{2\pi n}{T_0}\h$, and an energy width
$\cO(\h)$. We first propagate this Gaussian beam using the Schr\"odinger equation;  for the model operator $\tP_2(\h)$ the propagation is explicit, and preserves the factorization of variables:
\begin{equation}\label{e:quadrat}
\psi_{t}(s,x) \defeq [e^{-i t \tP_2/\h}\,\psi_0](s,x) = e^{-itE_\h/\h}\, \frac{e^{-t\llambda/2}}{(\pi \h)^{1/4}}\,e^{2 \pi i n s}\,  \exp\big(-\frac{x^2}{2e^{2t\llambda}\h}\big)\,.
\end{equation}
The longitudinal dependence $e^{2 \pi i n s}$ of the state is unchanged, but its transverse part is now a Gaussian of spatial width $\sim e^{t\llambda}\h^{1/2}$, which grows exponentially when $t>0$ (recall that the unstable manifold is the $x$-axis). The state remains localized in a microscopic neighbourhood of $\gamma$ until the {\it Ehrenfest time} $T_{\llambda}=\frac{1}{2\llambda}| \log \h |$. Notice that, with respect to the full geodesic flow on $S^*M$, $T_{\llambda}$ is only the {\it local} Ehrenfest time associated with the orbit $\gamma$.

To control the evolution of $\psi_0$ with the full QNF Hamiltonian $\tP(\h)$, we will use the results of Combescure and Robert \cite{CR97} who give an explicit procedure to calculate $e^{- i t \tP(\h)/\h} \psi_\h$ in terms of an asymptotic expansion in powers of $\h$ (see also \cite{Hag80}).  Although their expansion usually breaks down at the time $\sim\frac{1}{3}T_{\llambda}$, the special structure of the QNF $\tP(\h)$ allows here to keep the evolution under control until $T_\llambda$. Besides, up to this time the expansion is dominated by the state $\psi_t$ of \eqref{e:quadrat}.

3) \textbf{Time averaging}:  Obviously, each evolved state $\psi_t$ is a quasimode of $\tP_2$ with the same center and width $\cO(\h)$ as $\psi_0$. On the other hand, one can decrease the width of the quasimode by linearly combining these evolved states, namely averaging over the time.
Following the approach in \cite{VS05}, we consider the Ansatz:
\begin{equation} \label{ansatz}
\Psi_\h = \int_{-\infty}^{\infty} \chi(t/T) \, e^{it E_\h/h} \psi_t \,  dt\,,
\end{equation}
where $\chi\in C^\infty_c((-1,1))$ is a smooth cutoff, and $T>0$ is a (large) time. Both $\chi$ and $T$ will be optimized later. To compute the width of this quasimode, we must compare the norms of the states $(\tP-E_\h)\Psi_\h$ and $\Psi_\h$. Both computations rely on first estimating the overlaps $\la \psi_{t},\psi_{t'}\ra$ between evolved states. For our model Hamiltonian $\tP_2$ the overlap is explicit: $\la \psi_{t},\psi_{t'}\ra = \frac{1}{\sqrt{2\cosh((t-t')\llambda)}}$. For the evolution through the full QNF operator $\tP(\h)$, we will see that these overlaps decay essentially in the same way.
An integration by parts shows that
\begin{equation} \label{ibp1}
(\tP-E_\h)\int_{\IR} \chi(t/T) \,e^{it E_\h/h}\,\psi_t\, dt = 
\frac{i\h}{T}\int_{\IR} \chi'(t/T) \, e^{it E_\h/h}\,\psi_t dt\,.
\end{equation}
Hence, the larger the time $T$, the smaller the prefactor in the RHS. Because the evolved states remain localized until the time $T_\llambda$, we will choose the cutoff time $T = (1-\vareps')T_\llambda$ for some small $\vareps'>0$. The prefactor of the above right hand side explains the order of magnitude $C\,\frac{\h}{|\log \h|}$ of the width. To minimize the constant $C$ we also need to optimize the cutoff function $\chi$. This optimization completes the construction of fully localized quasimodes in Prop.~\ref{pro:quasim0}.

4) \textbf{Projections onto logarithmic intervals}: After constructing our logarithmic quasimode $\Psi_\h$ with center $E_\h$, width $C_\gamma \frac{\h}{| \log \h|}$, and which is microlocalized on $\gamma$, we project $\Psi_\h$ on an energy interval of width $2\vareps\h/|\log\h|$ centered near $E_\h$: this projection is then automatically a quasimode of width $\vareps\h/|\log\h|$. Elementary linear algebraic arguments show that, if the center $E'_\h$ of this interval is carefully chosen, the projected state is nontrivial, and has a positive weight on a small neighbourhood of $\gamma$.

\begin{rem}
The compactness of $M$ has not been used in the proof of Prop.~\ref{pro:quasim0}, but only in the final step of the proof of the Theorem, where we need the spectrum of $P(\h)$ to be discrete. This discrete spectrum is also the reason why we required $P(\h)$ to be semiclassically elliptic. Actually we need this discrete spectrum only in a neighbourhood of the energy $E_0$, so the result of Thm~\ref{thm:mainthm} can be generalized, for instance, to the case of Hamiltonians $p(x,\xi)\in S^m(T^*\IR^2)$ such that the energy layer $p^{-1}([E_0-\epsilon,E_0+\epsilon])$ is compact.
\end{rem}

\subsection*{Acknowldegements}
The work in this article began while S.E. was a resident at the Institut de Hautes Etudes Scientifiques, and was completed while being supported by an AMS-Simons Travel Grant.
S.N. was partially supported by the grant Gerasic-ANR-13-BS01-0007-02 from the Agence Nationale de la Recherche.


\section{Semiclassical preliminaries} \label{hprelim}
In this sections we recall the concepts and definitions from semiclassical Analysis we will need. 
The notations are drawn from the monographies \cite{DimSj99,Z12}. 

\subsection{Pseudodifferential operators on a manifold}
Recall that we define on $\mathbb{R}^{2d}$ the following class of symbols for $m \in \IR^2$:
\begin{eqnarray} \label{symbolclass}
S^{m}(\mathbb{R}^{2d}) &\defeq& \{a \in C^{\infty}(\mathbb{R}^{2d} \times (0, 1] ): |\partial^{\alpha}_x\partial^{\beta}_{\xi} a| \leq C_{\alpha,\beta} \langle\xi\rangle^{m-|\beta|} \}.
\end{eqnarray}
Symbols in this class can be quantized through the $\h$-Weyl quantization into the following pseudodifferential operators acting on $u\in \cS(\IR^d)$:
\begin{equation}
\Oph^{w}(a)\,u(x)\defeq
\frac{1}{(2\pi\hbar)^d}\int_{\mathbb{R}^{2d}}e^{\frac{i}{\hbar}\langle x-y,\xi\rangle}\,a\big(\frac{x+y}{2},\xi;\hbar\big)\,u(y)dyd\xi\,.
\end{equation}
One can adapt this quantization procedure to the case of the phase space $T^*M$, where $M$ is a smooth compact manifold of dimension $d$ (without boundary). Consider a smooth atlas $(f_l,V_l)_{l=1,\ldots,L}$ of $M$, where each $f_l$ is a smooth diffeomorphism from $V_l\subset M$ to a bounded open set $W_l\subset\IR^{d}$. To each $f_l$ correspond a pullback $f_l^*:C^{\infty}(W_l)\rightarrow C^{\infty}(V_l)$ and a symplectic diffeomorphism $\tilde{f}_l$ from $T^*V_l$ to $T^*W_l$: 
$$
\tilde{f}_l:(x,\xi)\mapsto\left(f_l(x),(Df_l(x)^{-1})^T\xi\right).
$$
Consider now a smooth partition of unity $(\phi_l)$ adapted to the previous atlas $(f_l,V_l)$. 
That means $\sum_l\phi_l=1$ and $\phi_l\in C^{\infty}(V_l)$. Then, any observable $a$ in $C^{\infty}(T^*M)$ can be decomposed as: $a=\sum_l a_l$, where 
$a_l=a\phi_l$. Each $a_l$ belongs to $C^{\infty}(T^*V_l)$ and can be pushed to a function $\tilde{a}_l=(\tilde{f}_l^{-1})^*a_l\in C^{\infty}(T^*W_l)$.
We may now define the class of symbols of order $m$ on $T^*M$ (after slightly abusing notation and treating $(x, \xi)$ as coordinates on $T^*W_l$)
\begin{eqnarray}
\label{pdodef}  S^{m}(T^*M)  &\defeq& \{a \in C^{\infty}(T^*M \times (0, 1] ): a=\sum_l a_{l},\ \  \text{ such that } \\ 
\nonumber && \tilde{a}_l\in S^m(\IR^{2d})\quad\text{for each }l\}.
\end{eqnarray}
This class is independent of the choice of atlas or smooth partition. 
For any $a\in S^{m}(T^{*}M)$, one can associate to each component $\tilde{a}_l\in S^{m}(\mathbb{R}^{2d})$ its Weyl quantization $\Oph^w(\tilde{a}_l)$, which acts on functions on $\IR^{2d}$.
To get back to operators acting on $M$, we consider smooth cutoffs $\psi_l\in C_c^{\infty}(V_l)$ such that $\psi_l=1$ close to the support of $\phi_l$, and define the operator:
\begin{equation}
\label{pdomanifold}\Oph(a)u \defeq 
\sum_l \psi_l\times\left(f_l^*Op_{\hbar}^w(\tilde{a}_l)(f_l^{-1})^*\right)\left(\psi_l\times u\right),\quad u\in C^{\infty}(M)\,.
\end{equation}
This quantization procedure maps (modulo smoothing operators with seminorms $\mathcal{O}(\hbar^{\infty})$) symbols $a\in S^{m}(T^{*}M)$ onto the space $\Psi^{m}_\hbar(M)$ of semiclassical pseudodifferential 
operators of order $m$. The dependence in the cutoffs $\phi_l$ and $\psi_l$ only appears at order 
$\hbar\Psi^{m-1}_\hbar$ (\cite[Thm 18.1.17]{Ho85} or \cite[Thm 9.10]{Z12}), so that the principal symbol map $\sigma_0:\Psi^{m}_\hbar(M)\rightarrow S^{m}(T^*M)/\h S^{m-1}(T^*M)$ is 
intrinsically defined. Most of the rules and microlocal properties (for example the composition of operators, the Egorov and Calder\'on-Vaillancourt Theorems) that hold on $\mathbb{R}^{d}$ can be extended to the manifold case.

An important example of a pseudodifferential operator is the semiclassical Laplace-Beltrami operator $P(\h)=-\frac{\h^2}{2} \Delta_g$.  In local coordinates $(x; \xi)$ on $T^*M$, the operator can be expressed as $Op_h^w \big( |\xi|^2_g + \h (\sum_j b_j(x) \xi_j + c(x)) + \h^2 d(x) \big)$ for some functions $b_j,c,d$ on $M$.  In particular, its semiclassical principal symbol is the function $|\xi |^2_g \in S^{2}(T^*M)$.  Similarly, the principal symbol of the Schr\"odinger operator 
$-\frac{\h^2}{2} \Delta_g + V(x)$ (with $V\in C^\infty(M)$) is $|\xi |^2_g + V(x) \in S^{2}(T^*M)$.

We will need to consider a slightly more general class of symbols than the class \eqref{symbolclass}.  Following \cite{DimSj99}, for any $0 \leq \delta < 1/2$ we introduce the symbol class
\begin{eqnarray} \label{singsymbolclass}
\nonumber S^{m}_{\delta}(\mathbb{R}^{2d}) &\defeq& \{a \in C^{\infty}(\mathbb{R}^{2d} \times (0, 1] ): |\partial^{\alpha}_x\partial^{\beta}_{\xi} a| \leq C_{\alpha,\beta} \h^{-\delta|\alpha + \beta|}\langle\xi\rangle^{m-|\beta|} \}.
\end{eqnarray}
These symbols are allowed to oscillate more strongly when $\h\to 0$.
All the previous remarks regarding the case of $\delta=0$ transfer over in a straightforward manner. This slightly ``exotic'' class of symbols can be adapted on $T^*M$ as well. For more details, see \cite[Section 3]{DG13} or \cite[Section 14.2]{Z12}.

\section{Normal form around a hyperbolic trajectory}\label{s:QNF}

Although the 2-dimensional quantum Birkhoff normal form (abbreviated by QNF) that we present below from \cite{Sj02} was proved in the setting of real analytic Hamiltonians on $T^*\IR^{2}$, its proof directly transfers into the $C^{\infty}$ setting, and we will present a sketch of it for the sake of the reader. 

Let us recall our setting. $(M,g)$ is a smooth surface without boundary and $P(\h)$ is a pseudodifferential operator on $M$ with principal symbol $p\in S^m(T^*M)$ assumed to be independent of $\h$, and formally selfadjoint on $L^2(M,dg)$. We assume that for a {\it regular} energy level $E_0$ (meaning that $dp$ does not vanish on the energy shell $p^{-1}(E_0)$), the Hamiltonian flow $exp(tH_p)$ admits a hyperbolic periodic orbit $\gamma=\gamma(E_0) \subset p^{-1}(E_0)$, with period $T(E_0)$. The hyperbolicity implies that this orbit is isolated in $p^{-1}(E_0)$, and that it belongs to a smooth family $\{\gamma(E)\}$ of hyperbolic orbits, with which each inside of $p^{-1}(E)$. The linearized Poincar\'e map of $\gamma$ has eigenvalues $\Lambda(E_0),\ \Lambda(E_0)^{-1}$ with $|\Lambda(E_0)|>1$; $\Lambda(E_0)$ is positive (resp. negative) if the local unstable manifold of $\gamma$ is orientable (resp. nonorientable). In both cases we denote $|\Lambda(E_0)|=\exp(T(E_0)\lambda(E_0))$, so that $\lambda(E_0)$ is the expansion rate per unit time along the unstable direction. This parameter will play an important role in the following.

Denoting by $\alpha$ the Liouville 1-form on $T^*M$, which reads $\alpha=\sum_{j=1}^2 \eta_j\,dy_j$ in local coordinates $(y_i,\eta_i)$ it. We call 
$$
\varphi=\varphi(E_0)\defeq \int_{\gamma(E_0)}\alpha,\qquad \text{the action of the orbit $\gamma(E_0)$.}
$$ 
\begin{prop} \label{prop:qbnf} \cite{Sj02}
i) There exists a Birkhoff normal form around the orbit $\gamma(E_0)$. Namely, for any integer $N>0$, there exist local symplectic coordinates $(s, \tau, x, \xi) \in T^*(S^1 \times \IR)$ and a smooth local canonical transformation 
$$
\kappa=\kappa_N : \neigh \big( \gamma_0\defeq\{(s,0,0,0),\,s\in S^1 \}, T^*(\IT \times \R) \big) \rightarrow \neigh\big(\gamma(E_0), T^*M \big)\,,
$$
single-valued if the unstable manifold of $\gamma(E_0)$ is orientable, and otherwise double-valued with $\kappa(s-1,\tau,x, \xi) = \kappa(s, \tau, -x, -\xi)$, such that 
\begin{equation}\label{eq:BNF}
\tp_N\defeq p \circ \kappa = f(\tau) + \llambda(\tau)x \xi + q^{(N)}(\tau, x \xi) + \mathcal{O}((x,\xi)^{N+1})\,,\qquad q^{(N)}(\tau, x \xi) = \sum_{4 \leq 2 \alpha \leq N} q_{\alpha}(\tau)(x \xi)^{\alpha}\,.
\end{equation}
Here $f(\tau)=E_0+\frac{\tau}{T_0}+\cO(\tau^2)$, and to each value $\tau\in [-\eps,\eps]$ corresponds a periodic orbit of energy $E(\tau)$, with $E(0)=E_0$. By a slight abuse of notation we denote $\lambda(\tau)=\lambda(E(\tau))$ the unstable expansion rate for the orbit $\gamma(E)$. The remainder means that $|\cO((x,\xi)^{N+1})|\leq C (|x|+ |\xi|)^{N+1}$ in the indicated neighborhood of $T^*(\IT \times \IR)$.  

ii) Recall that $\varphi$ is the action of the orbit $\gamma(E_0)$. Given $\h\in (0,1]$, let $\cS^{o/no}_{\phi/\h}(\IT \times \R)$ be the space of smooth functions $u(t,x)$ on $\R \times \R$, in the Schwartz class in the second variable, with the periodicity property $u(s+1,x) = e^{-i\phi/\h}u(s,x)$ in the orientable case, resp. $u(s+1,x) = e^{-i\phi/\h}u(s,-x)$ in the non-orientable case.

There exists a corresponding quantum Birkhoff normal form, namely a semiclassical Fourier integral operator $U_N: \cS^{o/no}_{\varphi/\h}(\IT \times \R) \rightarrow C^\infty(M)$ quantizing $\kappa_N$ microlocally near $\gamma_0$, microlocally unitary near $\gamma(E_0)\times\gamma_0$ (using the natural $L^2(\IT\times \IR)$ structure). This operator conjugates $P=P(\h)$ to the following Quantum Normal Form operator: 
\begin{equation}\label{e:QNF-conjugation}
 U_N^{*}\, \,P\, U_N = P^{(N)} + R_{N+1}\,.
\end{equation}
Here $P^{(N)}$, $R_{N+1}$ are pseudodifferential operators on $\cS_{\varphi/\h}(\IT \times \R)$, microlocally supported near $\gamma_0$, with Weyl symbols\footnote{Since $\IT=\IR/\IZ$, the Weyl quantization is intrinsically defined on $T^*(\IT\times\IR)$ by pullback from the quantization on $T^*\IR^2$.} 
\begin{equation}\label{e:higher-QNF}
p^{(N)}(\tau, x, \xi; \h) = \sum_{j=0}^{N} \h^j\,p_j^{(N)}(\tau, x, \xi),\qquad r_{N+1}(s,\tau, x, \xi; h) = \cO\big((\h,x,\xi)^{N+1}\big)\,,
\end{equation}
where $p_0^{(N)}=\tp_N$ above, and the higher order terms are of the same form.
Moreover, $p^{(N+1)}(\tau, x, \xi; \h) - p^{(N)}(\tau, x, \xi; \h) = \cO\big( (\h,x,\xi)^{N+1} \big)$.
\end{prop}
\begin{proof} (Sketch, after \cite{GS87,Sj02})

{\it i)} We will only consider the orientable case. The first step of the proof lies in the construction of suitable symplectic coordinates $\kappa_N (s,\tau;x,\xi) = (y,\eta)$ such that the classical Hamiltonian $p=\sigma_0(P)$ in these coordinates reads as in \eqref{eq:BNF}. That is, we construct the classical normal form near the orbit $\gamma(E_0)$. Let us recall this construction.

As mentioned above, the hyperbolicity of $\gamma(E_0)$ implies 
that in some range $E\in [E_0-\epsilon_0,E_0+\epsilon_0]$, the flow admits a hyperbolic orbit $\gamma(E)$, so that 
$\Gamma \defeq \bigcup_{| E - E_0| \leq \epsilon_0}\gamma(E)$
forms a smooth symplectic submanifold.  
If we call $\Gamma_{\pm}(E)$ the local unstable/stable manifolds of $\gamma(E)$, which are smooth immersed Lagrangian leaves, their unions
$\Gamma_{\pm} \defeq \bigcup_{| E - E_0| < \epsilon_0} \Gamma_{\pm} (E)$
are smooth involutive outgoing/incoming submanifolds of codimension 1. One can use the implicit function theorem to define a smooth function $\xi$ near $\gamma(E_0)$ such that locally $\Gamma_+=\{\xi=0\}$, and a symplectically conjugate coordinate $x$ with $x\restriction_{\Gamma_-}=0$. We may also setup canonical coordinates $(s,\tau)$ on $\Gamma$, with $s$ multivalued, such that each orbit $\gamma(E)$ corresponds to a circle $\{\tau=\tau(E),\ s\in[0,1)\}$. This set of coordinates on $\Gamma$ can be extended to a neighbourhood of $\Gamma$, such as to produce a system of canonical coordinates $(y,\eta)=\kappa_2(s,\tau,x,\xi)$ with $\tau=\tau(E)$ constant on each submanifold $\Gamma_{\pm}(E)$. In these coordinates near $\gamma(E_0)$, the Hamiltonian is of the following form:
$$
p\circ \kappa_2(t,\tau;x,\xi)=f(\tau)+\llambda(\tau)x\xi + \cO\big((x,\xi)^3\big)\,,
$$
where the function $\tau\mapsto f(\tau)$ is the inverse function of $E\mapsto \tau(E)$.
In order to normalize the higher order terms in $(x,\xi)$, we
can iteratively construct functions $(G_j(t,\tau,x,\xi))_{j\geq 3}$ (which are homogeneous polynomials of degree $j$ in $(x,\xi)$, satisfying certain solvable ODEs), and use the Hamiltonian flows $e^{tH_{G_j}}$ they generate, such as to obtain, at the order $N$, the normal form:
$$ 
\tp_N=p\circ\kappa_2 \circ e^{H_{G_3}} \circ \dots \circ e^{H_{G_N}} = 
f(\tau) + \llambda(\tau)x \xi + \sum_{2 \leq \alpha \leq N} q_{\alpha}(\tau) (x\xi)^N + \mathcal{O}((x,\xi)^{N+1})\,,
$$
which is the desired Birkhoff normal form. The induced change of coordinates $\kappa_N=\kappa_2 \circ e^{H_{G_3}} \circ \dots \circ e^{H_{G_N}}$ maps a neighbourhood of $\gamma_0\in T^*(\IT\times\IR)$ to a neighbourhood of $\gamma(E_0)\in T^*M$. Notice that $\kappa_N$ and $\kappa_{N+1}$ have the same Taylor expansion in the coordinates $(x,\xi)$ up to order $N$.

\medskip

{\it ii)} One can adapt this normal form construction to the quantum framework (a general discussion on quantum normal forms can be found in \cite[Chap.12]{Z12}). The task is to construct a semiclassical Fourier Integral Operator (FIO) $\cU_0$ quantizing the canonical transformation $\kappa_N$. The construction of the phase function of $\cU_0^*$ shows that this operator maps the space $C^\infty(M)$ to the space $\cS_{\varphi/\h}(\IT\times \IR)$ of functions which are $\IZ$-periodic functions in the variable $s$, twisted by the phase $e^{i\varphi/\h}$ (see for instance \cite[Eq.(2.9)]{GS87}). The FIO $\cU_0$ is elliptic from a neighbourhood of $\gamma_0$ to a neighbourhood of $\gamma(E_0)$, and can be constructed such as to be microlocally unitary between two such neighbourhoods.

Conjugating our quantum Hamiltonian $P(\h)$ with $\cU_0$, we obtain a pseudodifferential operator $\tP_0(\h)\defeq U_0^* P(\h) U_0$ on $\cS_{\varphi/\h}(\IT\times \IR)$, whose principal symbol is of the form $\tp_N(s,\tau,x,\xi)$.
The subprincipal symbol of $\tP_0(\h)$ is, a priori, an arbitrary function of $(s,\tau;x,\xi)$. By iteratively solving a sequence of transport equations, we may correct the symbol of the FIO $\cU_0$ into a FIO $\cU_N$, so that the symbol of the conjugated operator $\tP_N(\h)=U_N^* P(\h) U_N$ takes the form \eqref{e:higher-QNF}. 

Our main point is that the above construction, originally presented in the case of analytic symbols and operators on the Euclidean space \cite{GS87,Sj02}, can be generalized to the case of smooth objects on a smooth manifold (the FIO $\cU_0$ can be constructed with a real valued phase function). Since the construction is local near $\gamma(E_0)$, topological properties of the manifold $M$ (e.g. the homotopy class of $\gamma(E_0)$) do not come into play.
\end{proof}

\begin{rem}\label{r:higher-dims}
Similar quantum normal forms were established in higher dimension. The references \cite{Gui96,Zel98} are specific to the homogeneous microlocal setting. A higher dimensional formulation of the QNF we use was given in \cite{GuiPaul09}, albeit the proof for the case of a hyperbolic trajectory is only sketched. A similar QNF was also derived in \cite{Chr07}, but at the cost of a remaining elliptic factor depending on the longitudinal variable.

Using these higher dimensional QNF, one can probably generalize the quasimode construction we are presenting below. 
The main extra difficulty is to compute the corresponding square norm of the quasimode $\Psi$ as performed in Lemma~\ref{l:norm-qmode}. This computation requires to understand well the overlaps between evolved coherent states in higher dimension (the generalization of \eqref{e:phi0Dphi0}), which will depend on the spectrum and Jordan structure of the linearized Poincar\'e map of $\gamma(E_0)$. 
Since obtaining sharp constants is one objective of the present article, we restrict ourselves to the 2-dimensional case.
\end{rem}

\section{Propagation of a Gaussian wavepacket at the hyperbolic fixed point} \label{propagsec}

The construction of our quasimode will be performed on the Quantum Normal Form side, that is on the model space $\IT\times \IR$. As explained in $\S$\ref{s:outline}, it will be based on an initial state which is a plane wave in the longitudinal direction, and Gaussian along the transverse direction. We must nevertheless take into account that the QNF operator acts on twisted-periodic functions. From the action $\varphi(E_0)$ and $\h\in (0,1]$ we setup the number 
$$
\varphi_\h\defeq 2\pi \big[ \frac{\varphi(E_0)}{2\pi\h}\big],\quad \text{where $[s]$ represents the integral part of $s\in\IR$. }.
$$
It satisfies $e^{i\varphi_\h}=e^{i\varphi(E_0)/\h}$, so the spaces of twisted periodic functions can be denoted by $\cS^{o/no}_{\varphi_\h}(S^1\times \IR)$. 
With this twist being taken into account, our initial state will be of the same form as in \eqref{e:psi0}:
\begin{equation}\label{e:initial}
\psi_0(s,x)\defeq e^{i(2 \pi n -\varphi_\h)s}\, \frac{1}{(\pi \h)^{1/4}}\, \exp\big(-\frac{x^2}{2\h}\big),\qquad \text{for some arbitrary }n\in\Z.
\end{equation}
Due to the parity of $x\mapsto e^{-x^2/2\h}$, this function belongs to both spaces $\cS^{o/no}_{\varphi_\h}(\IT\times \IR)$, and it will allow us to treat both the orientable and nonorientable cases.

We will select the index $n$ to be uniformly bounded when $\h\to 0$, so that 
$$
\Oph^w(f(\tau))\psi_0=f(\h D_s) \psi_0 = f(\h(2\pi n - \varphi_\h))\psi_0 \quad\text{satisfies }f(\h(2\pi n - \varphi_\h))=E_0+\cO(\h)\,.
$$
We will evolve the state $\psi_0$ through the Schr\"odinger equation generated by the QNF operator $P^{(N)}$ described in Prop.~\ref{prop:qbnf},{\it ii)}: our task will now be to describe the states
$$
\psi^{(N)}_t = e^{-iP^{(N)}t/\h}\,\psi_0,\quad \text{for times }|t|\leq C\,|\log\h|\,.
$$
The Weyl symbol $p^{(N)}$ of $P^{(N)}$ can be written 
\begin{equation}
p^{(N)}(\tau, x, \xi; h) = \sum_{\alpha=0}^{N} q^{\alpha}(\tau;\h)\,(x\xi)^\alpha\,,
\end{equation}
where the symbols $ q^{\alpha}(\tau,\h)$ expand as
$$
q^\alpha(\bullet;\h)=\sum_{i=0}^N \h^i\,q^\alpha_i(\bullet)\,,\quad\text{with the special values } q^0_0(\tau)=f(\tau),\quad q^1_0(\tau)=\llambda(\tau)\,.
$$ 
The major advantage of this QNF is the separation between the variables $(s,\tau)$ and $(x,\xi)$. As a result, the Schr\"odinger evolution can be reduced to a family of 1-dimensional problems. Indeed, since $e^{i(2 \pi n -\varphi_\h)s}$ is an eigenfunction of $\h D_s$, the state $\psi_t$ can be factorized into longitudinal and transversal parts,
$$
\psi^{(N)}_t(s,x) = e^{-itq^0/\h}\, e^{i(2 \pi n -\varphi_\h)s} \varphi^{(N)}_t(x)\,,\quad\text{where we use the shorthand notation } 
q^0 = q^0((2\pi n-\varphi_\h)\h;\h)\,.
$$
Notice that $q^0=E_0+\cO(\h)$. 
The transversal part $\varphi^{(N)}_t$ satisfies the 1D Schr\"odinger equation
\begin{align}
i\h \partial_t \varphi^{(N)}_t &= Q^{(N)}(\h)\,\varphi^{(N)}_t,\quad Q^{(N)}=\Oph^w(q^{(N)}),\nonumber\\
q^{(N)}(x,\xi;\h) &= \sum_{\alpha=1}^{N} q^{\alpha}\,(x\xi)^\alpha\,,\qquad \text{where we took }\ \ 
q^\alpha = q^\alpha((2\pi n-\varphi_\h)\h;\h)\,.\label{e:q^(N)}
\end{align}
We thus end up with analyzing the evolution of the 1D Gaussian state $\varphi_0(x)=(\pi \h)^{-1/4}\, e^{-\frac{x^2}{2\h}}$ under the effective 1D Schr\"odinger equation. 

Let us truncate the symbol $q^{(N)}$ to the quadratic order, that is keep from \eqref{e:q^(N)} 
the term
\bequ
q^{(N)}_q = q^1\,x\xi\,,\qquad \text{where}\ \   q^1=q^1((2\pi n-\varphi_\h)\h;\h)=\llambda(0)+\cO(\h)\,.
\eequ
The evolved state through this quadratic operator is easily expressed:
\bequ\label{e:quad2}
\exp\big(-\frac{it}{\h} Q^{(N)}_q \big) \varphi_0 = \cD_{tq^1}\varphi_0\,, 
\eequ
and we used the unitary dilation operator $\cD_\beta:L^2(\IR)\to L^2(\IR)$:
\bequ
\cD_\beta u (x)\defeq \exp(-i\beta \Oph^w(x\xi)/\h) u (x) = e^{-\beta/2} u (e^{-\beta} x)\,.
\eequ
The state $\cD_{tq^1}\varphi_0$ is known in the literature as a {\it squeezed coherent state} \cite{CR97}: it is a Gaussian state, but with a different width from that of $\psi_0$. This state is still centered at the origin in phase space, $(x,\xi)=(0,0)$, which is a fixed point for the classical evolution generated by the Hamiltonian $q^{(N)}_0$.

For $N\geq 2$, the operator $Q^{(N)}$ includes a nonquadratic part 
\bequ\label{e:q_nq}
q^{(N)}_{nq}\defeq q^{(N)}-q^{(1)} = \sum_{\alpha=2}^{N}q^\alpha (x\xi)^\alpha\,.
\eequ
We show below that, in the semiclassical limit, this nonquadratic component of $Q^{(N)}$ will induce small corrections to the state \eqref{e:quad2} evolved through the quadratic part. To justify this fact we will use a Dyson expansion, like in the general treatment of \cite[Section 3]{CR97}, which considered the evolution of coherent states through arbitrary Hamiltonians. Here, the special form of $q^{(N)}$ will facilitate our task and will produce smaller remainder terms. 
Remember that the state \eqref{e:quad2} is a squeezed coherent state. The corrections due to the nonquadratic part $Q^{(N)}_{nq}$ are linear combinations of {\it squeezed excited states}.


\medskip

{\bf Squeezed excited states} The initial coherent state $\varphi_0$ is the ground state of the standard 1D Harmonic oscillator $(\h D_x)^2+x^2$. We will call $(\varphi_m)_{m\geq 1}$ the $m$-th excited states, which are obtained by iteratively applying to $\varphi_0$ the ``raising operator'' $a^*\defeq \Oph(\frac{x-i\xi}{\sqrt{2\h}})$ and normalizing: 
\bequ\label{e:excited}
\varphi_m = \frac{(a^*)^m}{\sqrt{m!}} \varphi_0 \Longrightarrow \varphi_m(x) = \frac{1}{(\pi \h)^{1/4}2^{m/2}\sqrt{m!}}\,H_m(x/\h^{1/2})\, e^{-\frac{x^2}{2\h}}\,,
\eequ
where $H_m(\cdot)$ is the $m$-th Hermite polynomial. 
By applying the unitary dilation operator $\cD_\beta$,
we obtain a family of squeezed excited states $(\cD_\beta \varphi_m)_{m\geq 1}$.

Following the strategy of \cite{CR97}, we will show the following approximate expansion for $e^{-it Q^{(N)}/\h ) \varphi_0}$ in terms of squeezed excited states.
\begin{prop} \label{prop:propag}
For every $l\in \IN$, there exists a constant $C_l > 0$ and time dependent coefficients $c_{m}(t, \h)\in\C$ for $0\leq m \leq 2l$, such that the following estimate holds for any $\h\in (0,1]$:
\begin{equation} \label{remest}
\forall t\in\IR,\qquad \| e^{itQ^{(N)}/\h} \varphi_0 - \cD_{tq^1}\varphi_{0} - \sum_{m=0}^{2l} c_{m}(t,\h) \cD_{tq^1}\varphi_{2m} \| \leq C_l\,(|t|\h)^{l+1}\,.
\end{equation}
The coefficients $c_{m}(t,\h)$ are polynomials in $(t,\h)$, with degree at most $l$ in the variable $t$. 
Besides, $c_0(t,\h)=\cO_t(\h)$, $c_1(t,\h)=\cO_t(\h^2)$, $c_{m}(t,\h) = \cO_t(\h^{[(m+1)/2]})$ for $2\leq m\leq 2l$. 
\end{prop}
\begin{rem}
The result presented in \cite[Thm 3.1]{CR97} concerned the evolution of $\varphi_0$ w.r.t. an arbitrary quantum Hamiltonian; in this general case the remainder usually grows exponentially with the time. The present {\it polynomial} growth of the remainder relies on the normal form structure of $Q^{(N)}$, namely the fact that $Q^{(N)}_{nq}$ is a sum of powers of $\Oph(x\xi)$, which commutes with the quadratic evolution.
\end{rem}
\begin{proof}
Like in \cite{CR97}, we want to compare the full evolution $\varphi^{(N)}_t= U(t)\varphi_0 = e^{-itQ^{(N)}/\h}\varphi_0$ with the quadratic one, $U_q(t)\varphi_0 = e^{-itQ^{(N)}_q/\h}\varphi_0$. The comparison is based on the Dyson expansion
\bequ
U(t) - U_q(t) = \frac{1}{i\h} \int_{0}^t U(t-t_1)\,Q^{(N)}_{nq}\,U_q(t_1) \, dt_1\,.
\eequ
This is the first order Dyson expansion (case $l=0$).
Let us show that $Q^{(N)}_{nq}$ commutes with $ Q^{(N)}_{q}$. Indeed, each operator $\Oph^w( (x\xi)^\alpha)$ is a linear combination of powers of $\Oph^w(x\xi)$:
\begin{lem}\label{l:x-xi}
For any $\alpha\geq 2$, there exists absolute constants $(c_{\alpha,k})_{1\leq k\leq [\alpha/2]}$ such that
\bequ\label{e:xxi^alpha}
\Oph^w( (x\xi)^\alpha ) = (\Oph^w(x\xi))^\alpha + \sum_{k=1}^{[\alpha/2]} c_{\alpha,k}\h^{2k}  (\Oph^w(x\xi))^{\alpha-2k}\,.
\eequ
\end{lem}
This Lemma can be proved by induction, using the fact that the $\hbar$-expansion of the Moyal product $(x\xi)^\alpha \#_\hbar (x\xi)$ terminates at the order $\hbar^2$.
As a result, $Q^{(N)}_{nq}$ commutes with $Q^{(N)}_{q}$ and therefore with $U_q(t_1)$:
\bequ\label{e:dyson1}
U(t) - U_q(t) = \frac{1}{i\h} \int_{0}^t U(t-t_1)\,U_q(t_1)\,Q^{(N)}_{nq} \, dt_1\,.
\eequ
We want to apply this operator to $\varphi_0$. We can represent $Q^{(N)}_{nq}\varphi_0$ as a linear combination of excited coherent states. Indeed, using the expression $\Oph(x\xi)= i\h ((a^*)^2 - a^2)$ in terms of the raising and lowering operators $a^*,a$, we may write:
$$
\Oph^w( (x\xi)^\alpha ) = \h^\alpha \Big(i^\alpha((a^*)^2 - a^2)^\alpha + \sum_{k=1}^{[\alpha/2]} i^{\alpha-2k} c_{\alpha,k} ((a^*)^2 - a^2)^{\alpha-2k}\Big)\,.
$$
Using the commutation relation $[a,a^*]=1$, the RHS can be rewritten in ``normal ordering'', that is as a linear combination of terms $(a^*)^{2\beta}a^{2\gamma}$, with $0\leq \beta,\gamma\leq \alpha$. Since $\varphi_0$ satisfies $a\varphi_0=0$, we only keep the terms with $\gamma=0$. Using the definition \eqref{e:excited} of the excited coherent states $\varphi_m$, we get the following expression:
\begin{lem}\label{lem:xxi^alpha}
For any $\alpha\geq 2$, there exists coefficients $\{d_{\alpha,k},\ 0\leq k\leq [\alpha/2]\}$ such that
\bequ
\Oph( (x\xi)^\alpha )\varphi_0 = \h^\alpha \sum_{k=0}^{[\alpha/2]} d_{\alpha,k}\, \varphi_{2\alpha-4k}\,.
\eequ
\end{lem}
From \eqref{e:q_nq}, the state $Q^{(N)}_{nq}\varphi_0$ is a linear combination of $\{\varphi_{2m},\ 0\leq m\leq N\}$. 
By inspection, we see that the coefficients in front of $\varphi_{2m}$ have the following orders in $\h$:
$$
Q^{(N)}_{nq}\varphi_0 = \cO(\h^2) (\varphi_0 + \varphi_4 ) + \cO(\h^3) (\varphi_2 + \varphi_6 ) + \cO(\h^4)\varphi_8 +\cdots + \cO(\h^N)\varphi_{2N}\,.
$$
From \eqref{e:quad2}, the action of $U_q(t_1)$ results in a multiplication by  $e^{-it_1q^0/\h}$, and the replacement of $\varphi_{2m}$ by the squeezed states $\cD_{t_1q^1}\varphi_{2m}$. Taking into account the factor $i/\h$ in front of the integral \eqref{e:dyson1} and the unitary of $U(t)$, we get the simple estimate
$$
\|(U(t) - U_q(t))\varphi_0\|_{L^2(\IR)}\leq C\,t\h\,,\quad \forall \h\in (0,1],\ \forall t\in\IR\,.
$$

\medskip

In order to improve the description of $U(t)\varphi_0$, we shall expand $U(t) - U_q(t)$ into a Dyson expansion of higher order. For any order $l\geq 1$, this expansion reads:
\begin{multline*}
U(t) - U_q(t)=\sum_{j=1}^{l} \frac{1}{(i\h)^j} \int_{0}^t \int_{t_1}^t \dots \int_{t_{j-1}}^t U_q(t-t_j) Q^{(N)}_{nq}
U_q(t_j - t_{j-1})Q^{(N)}_{nq} \cdots Q^{(N)}_{nq} U_2(t_{1})  \, dt_1 \dots dt_j\\
+ \frac{1}{(i\h)^{l+1}} \int_{0}^t \int_{t_1}^t \dots \int_{t_{l}}^t U(t-t_{l+1}) Q^{(N)}_{nq}
U_q(t_{l+1} - t_{l})Q^{(N)}_{nq} \cdots Q^{(N)}_{nq} U_2(t_{1})  \, dt_1 \dots dt_{l+1}
\end{multline*}
Using the commutativity of $Q^{(N)}_{nq}$ with $U_q(s)$, this simplifies to
\bequ\label{e:dyson-l}
U(t) - U_q(t)=\sum_{j=1}^{l} \frac{t^j}{j!(i\h)^j} U_q(t) (Q^{(N)}_{nq})^j
+ \frac{1}{(i\h)^{l+1}} \int_{0}^t \frac{t_l^{l}}{l!} U(t-t_{l+1})\,
U_q(t_{l+1})\, ( Q^{(N)}_{nq} )^{l+1}\,dt_{l+1}\,.
\eequ
When applied to $\varphi_0$, each $j$-term on the RHS leads to a linear combination of squeezed excited states $\{\cD_{tq^1}\varphi_{2m},\ 0\leq m\leq Nj\}$, with the following estimates on the coefficients:
\begin{multline*}
\h^{-j}(Q^{(N)}_{nq})^j\varphi_0 = \cO(\h^{j})(\varphi_{4j}+\varphi_{4j-4}+\cdots+\varphi_0)+ \cO(\h^{j+1})(\varphi_{4j+2}+\varphi_{4j-2}+\cdots+\varphi_2)+\\ 
+\cO(\h^{j+2})\varphi_{4j+4}+\cdots+\cO(\h^{(N-1)j})\varphi_{2Nj}.
\end{multline*}
Hence, each $j$-term in the equation \eqref{e:dyson-l} has a norm
$$
\big\| \frac{t^j}{j!(i\h)^j} (Q^{(N)}_{nq})^j\varphi_0 \big\|\leq C_j\,(|t|\h)^j,\,
$$
where the implicit constant only depends on the coefficients $q^\alpha$ of $Q^{(N)}$.
For the same reasons, the remainder term in \eqref{e:dyson-l}, when applied to $\varphi_0$, has a norm bounded by 
\bequ\label{e:remain-l}
\Big\|\frac{1}{(i\h)^{l+1}} \int_{0}^t \frac{t_l^{l}}{l!} U(t-t_{l+1})\,U_q(t_{l+1})\, ( Q^{(N)}_{nq} )^{l+1}\varphi_0 \,dt_{l+1} \Big\|\leq C_l\,(|t|\h)^{l+1}\,.
\eequ
In the $j$-sum, each squeezed excited state $\cD_{tq^1}\varphi_{2m}$ has a coefficient given by a certain polynomial $c_m(t,\h)$. We won't need to analyze this polynomial in detail, but only give partial information. Because we will deal only with times $|t|\leq C|\log\h|$, the size of each polynomial $c_m(t,\h)$ will be guided (up to a logarithmic factor) by the term with the smallest $\h$-power. 

Beyond the principal term $U_q(t)\varphi_0$, the Dyson expansion will be of the form
$$
\cD_{tq^1}\big[\cO_t(\h) (\varphi_0 + \varphi_4) + \cO_t(\h^2)(\varphi_2+\varphi_6+\varphi_8)  + \cO_t(\h^3)(\varphi_{10}+ \varphi_{12}) + \cO_t(\h^4)(\varphi_{14} + \varphi_{16}) +\ldots\big]\,,
$$
where the coefficient in front of each $\varphi_m$ is a polynomial in $t$ which vanishes when $t=0$. In the range $3\leq j\leq l$, the terms are of the form $\cO_t(\h^j)(\varphi_{2(2j-1)}+\varphi_{4j})$. This ends the proof of the Proposition.
\end{proof}

After tackling the 1D evolution, we can now reconstruct the full evolved state 
\bequ\label{e:factor-psi_t}
\psi_t^{(N)}(s,x)\defeq e^{-itP^{(N)}/\h}\psi_0 (s,x) = e^{-itq^0/\h}\,
e^{i(2 \pi n -\varphi_\h)s}  \,\varphi_t^{(N)}(x)\,.
\eequ

\subsection{Microlocal support of the evolved state}
The expansion in Prop.~\ref{prop:propag} shows that the state $\psi^{(N)}_t$ is under control until polynomial times 
$|t|\sim\h^{-1}$. However, for such large times the coherent states $\cD_{tq^1}\varphi_{2m}$ will be very delocalized. Because the normal form is valid only in a small neighbourhood of $\gamma_0$ and it involves remainders $\cO((x,\xi)^{N+1})$, it is important to keep our states to be microlocalized in a microscopic neighbourhood of $\gamma_0$. For this reason, we are forced to bound the time of evolution in a precise logarithmic window.
Namely, we select a small $\vareps'\in(0,1)$, and for any $\h\in (0,1/2]$ we define the local Ehrenfest time
\begin{equation}\label{e:Ehrenfest}
T_{\vareps'}\defeq \frac{(1-\vareps')|\log\h|}{2\llambda}\,,
\end{equation}
where as above $\lambda=\lambda(0)$ is the expansion rate of $\gamma_0$.
\begin{prop}\label{p:localization1}
Take $\Theta\in C^\infty_c(T^*(\IT\times \IR))$ with $\Theta\equiv 1$ in a neighbourhood of $\gamma_0$, and denote its rescaling by $\Theta_{\alpha}(s,\tau;x,\xi)\defeq \Theta(s,\tau/\alpha;x/\alpha,\xi/\alpha)$. 
Then, for any power $M>0$, there exists $C_M>0$ such that
\bequ
\|[\Oph(\Theta_{\h^{\vareps'/3}})-I] \psi_t^{(N)}\|_{L^2} \leq
C_M\,\h^{M}\,,\qquad \h\in (0,1]\,,
\eequ
uniformly for times $t\in [-T_{\vareps'},T_{\vareps'}]$.
\end{prop}
The above property could be abbreviated as
$\|[\Oph(\Theta_{\h^{\vareps'/3}})-I] \psi_t^{(N)}\|_{L^2}= \cO(\h^\infty)$, uniformly in the time interval. Shortly speaking, the state $\psi_t^{(N)}$ is microlocalized in any $h^{\vareps'/3}$-neighbourhood of $\gamma_0$
\begin{proof}
We will check that all the squeezed excited states $\cD_{tq^1}\varphi_{2m}$, $0\leq m\leq M$, are microlocalized in the same neighbourhood of $(0,0)\in T^*\IR$, for $t$ in this time interval. 
The state
$\cD_{tq^1}\varphi_m(x)$ is a Gaussian of width $e^{q^1t}\h^{1/2}$, decorated by a polynomial factor. 
For $|t|\leq T_{\vareps'}$, this width takes values:
$$
e^{q^1t}\h^{1/2} \leq e^{q^1T_{\vareps'}}\h^{1/2} = e^{(\llambda+\cO(\h))T_{\vareps'}}\h^{1/2} = \h^{\vareps'/2+\cO(\h)}\,,
$$
hence it remains microscopic.
Consider a cutoff $\chi\in C^\infty_c(\IR,[0,1])$ supported in $[-2,2]$, equal to unity in $[-1,1]$, and define $\chi_\alpha(x)\defeq \chi(x/\alpha)$.
For $\alpha\geq\h^{\vareps'/3}$, a direct estimate of the Gaussian integral 
shows that for any $M>0$, 
$$ 
\|(\chi_{\alpha} -1)\cD_{tq^1}\varphi_m\|_{L^2}= \cO(\h^M)\,,
$$
uniformly in the time window.
This shows that $\cD_{tq^1}\varphi_m$ is microlocalized inside the strip $\{|x|\leq \h^{\vareps'/3}\}\subset T^*\IR$.
The semiclassical Fourier transform leaves the states $\varphi_m$ invariant (up to a constant factor), and inverts the dilation operator: $\cD_\beta (\cF_\h u ) = \cF_\h (\cD_{-\beta} u)$.
As a result, the above computation shows that $\cD_{tq^1}\varphi_m$ is also microlocalized inside the horizontal strip $\{|\xi|\leq \h^{\vareps'/3}\}$, uniformly for $|t|\leq T_{\vareps'}$.
These position and momentum microlocalizations imply that
$\cD_{tq^1}\varphi_m$ is microlocalized inside the square
$\{|x|,|\xi|\leq \h^{\vareps'/3}\}$. Hence, for any $\theta\in
C^\infty_c([-2,2]^2)$ with $\theta\equiv 1$ in $[-1,1]$, rescaled into
$\theta_{\alpha}(x,\xi)\defeq \theta(x/\alpha,\xi/\alpha)$, we get for
any index $m$ in a bounded range $[0,m_0]$:
$$
\|[\Oph(\theta_{\alpha})-I] \cD_{tq^1}\varphi_m\|_{L^2} = \cO(\h^M)\,,
$$
uniformly for $|t|\leq T_{\vareps'}$, index $m\in [0,m_0]$ and width $\alpha\geq \h^{\vareps'/3}$.

According to Prop.~\ref{prop:propag}  the 1D evolved state
$\varphi_t^{(N)}=e^{-itQ^{(N)}/\h}\varphi_0$ is a linear combination
of $l+1$ squeezed excited states, plus a remainder
$\cO(h^{l+1-\eps})$. Hence, taking $l=M$ and using the triangle inequality, we get: 
$$
\|[\Oph(\theta_{\alpha})-I] \varphi_t^{(N)}\|_{L^2} = \cO(\h^{M})\,.
$$
We now consider the full state $\psi_t^{(N)}$. Considering the cutoff $\Theta$ as in the statement, we choose an auxiliary cutoff $\tilde{\Theta}(s,\tau;x,\xi) = \chi(\tau)\theta(x,\xi)$, supported near $\gamma_0$, such that $\Theta\equiv 1$ near the support of $\tilde{\Theta}$. 
We rescale $\tilde{\Theta}$ as in the Proposition. If $\alpha\geq \hbar^{\vareps'/3}$, we observe that in the longitudinal variable, we have for $\hbar$ small enough:
$$
\chi_\alpha(\h D_s)(e^{i(2 \pi n -\varphi_\h)s})=\chi((2 \pi n -\varphi_\h)\h/\alpha)(e^{i(2 \pi n -\varphi_\h)s})=
e^{i(2 \pi n -\varphi_\h)s}\,.
$$
As a result, for $\hbar$ small enough we get
$$
\|[\Oph(\tilde{\Theta}_{\alpha})-I] \psi_t^{(N)}\|_{L^2} = \cO(\h^{M})
$$
uniformly for $|t|\leq T_{\vareps'}$ and width $\alpha\geq \h^{\vareps'/3}$. 

Let us finally check that the same estimate holds with the cutoff $\Theta_{\alpha}$, which is usually not of factorized form. 
From the support proprerties of $\Theta$ and $\tilde{\Theta}$, we have for any $\alpha$:
$$
\Theta_{\alpha}-1 =(1- \Theta_{\alpha}) (\tilde{\Theta}_\alpha-1)\,.
$$
For $\alpha\geq \hbar^{\vareps'/3}$ the symbol calculus in $S^0_{\vareps'/3}(\IT\times \IR)$ shows that this equality translates into
$$
\Oph(\Theta_{\alpha})-I =\big(I- \Oph(\Theta_{\alpha})\big) \big(\Oph(\tilde{\Theta}_\alpha)-I\big)+\cO(\h^\infty)_{L^2\to L^2}\,.
$$
\end{proof}
\begin{rem}
In the studies \cite{CR97,HJ99} on the long time evolution of coherent states, expansions of the type \eqref{remest} generally break down at the earlier time $T_{\lambda}/3$. The breakdown is related with the fact that the unstable manifold of the point where the state is centered is usually {\it curved} in the ambient coordinates $(y,\eta)$ used to define the coherent states. Since the evolved coherent state wants to expand {\it along this manifold}, it wants to curve too, which is incompatible with the elliptic shape of squeezed coherent states. This curvature effect  becomes critical around the time $T_{\lambda}/3$, leading to the breakdown of the expansion (coefficients with large $m$ become dominant). On the opposite, for our normal form Hamiltonian $q^{(N)}$, the unstable manifold of the fixed point at the origin is the horizontal line, which is not curved. The squeezed coherent states can perfectly spread along this line, explaining why our expansion does not develop any singularity up to polynomial times $t\sim \h^{-1+\eps}$.
\end{rem}

\section{$L^2$ norms and quasimode widths} \label{quasimodeorder}

Below we will construct a logarithmic quasimode for the QNF operator $P^{(N)}$, 
by averaging our evolved coherent states over the time. 
To begin with, let us compute the energy width of the 
initial Gaussian state $\psi_0$.
\begin{lem} \label{l:E-local1}
Consider the initial state $\psi_0(s,x)$ given in \eqref{e:initial}, with $n\in\IZ$ possibly $\h$-dependent, but  uniformly bounded when $\h\to 0$. 

Then if we take the energy
\bequ\label{e:E_h}
q^0=q^0(\h(2\pi n - \varphi_\h);\h) =E_0+\cO(\h)\,,
\eequ
we obtain the quasimode estimate
$$
\| (P^{(N)} - q^0)\psi_0\|_{L^2}=\llambda\sqrt{2}\,\h + \cO(\h^2)\,.
$$
\end{lem}
\begin{proof}
The separation between the $(s,x)$ variables allows us to replace this norm by a 1D norm:
$$
\|(P^{(N)} - q^0(\h))\psi_0\|_{L^2(\IT\times\IR)} = \|Q^{(N)}\varphi_0\|_{L^2(\IR)}\,.
$$
Considering the decomposition $Q^{(N)}=Q^{(N)}_q+Q^{(N)}_{nq}$, an explicit computation shows that
$$
Q^{(N)}_q\varphi_0 = q^1\Oph( (x\xi) )\varphi_0 = i\h q^1\sqrt{2}\varphi_2\,,
$$ 
Then, Lemma~\ref{lem:xxi^alpha} shows that each term $q^\alpha\Oph^w( (x\xi)^\alpha )\varphi_0$ in $Q^{(N)}_{nq}\varphi_0$ will produce a linear combination of excited states, of norms $\cO(\h^\alpha)$, with $\alpha\geq 2$, so these terms are subdominant with respect to the quadratic one. The property $q^1=\llambda+\cO(\h)$ achieves the proof.
\end{proof}
Hence $\psi_0$ is a quasimode centered at $q^0(\h)=E_0+\cO(\h)$ and of width $\sim C\h$. Obviously, this is also the case for each evolved state $\psi^{(N)}_t$. 

\subsection{Constructing the logarithmic quasimode for the normal form}

In this section we will construct a better quasimode for the QNF, using a time averaging procedure.

Take an arbitrary energy $E_\h=E_0+\cO(\h)$, a time $T>0$, a weight function $\chi\in C^\infty_c((-1,1),[0,1])$, and its rescaled version $\chi_T(t)\defeq \chi(t/T)$. Our quasimode is defined by
\bequ\label{eq:def-qmode}
\Psi_{\chi_T,E_\h} \defeq  \int_{\IR}\chi_T(t)\,e^{itE_\h/\h}\,\psi^{(N)}_t\,dt\,.
\eequ
It is important to note that this state is not normalized. In order to compute its energy width, we will first need to compute its $L^2$ norm. 
\begin{lem}\label{l:norm-qmode}
For $C>0$ and $\h\in(0,1]$, we consider a semiclassically large averaging time $1\leq T=T_\h\leq C|\log\h|$. 

Then the square norm of the state $\Psi_{\chi_T,E_\h}$ takes the form
$$
\|\Psi_{\chi_T,E_\h}\|^2 = T\,S_1(\llambda,(E_\h-q^0)/\h)\,\|\chi\|_{L^2}^2\, \big(1 + \cO(1/T)\big)\,,
$$ 
where $S_1(\bullet,\bullet)$ is a positive function given in \eqref{e:S_1}, and $q^0$ is the energy \eqref{e:E_h}.
\end{lem}
\begin{proof}
Like in the previous section, the factorized form of $\psi^{(N)}_t$ shows that the above norm is equal to the $L^2(\IR)$-norm of the 1D state
\bequ\label{e:Phi_chi}
\Phi_{\chi_T,E_\h}, \defeq  \int_{\IR}\chi_T(t)\,e^{it(E_\h-q^0)/\h}\,\phi^{(N)}_t\,dt\,.
\eequ
Recall from Prop.~\ref{prop:propag} that the evolved state $\phi^{(N)}_t$ is a combination of squeezed excited coherent states:
$$
\phi^{(N)}_t = \cD_{tq^1}\varphi_0 + \sum_{m=0}^{2l} c_{m}(t,\h) \cD_{tq^1}\varphi_{2m} + R_l\,,
$$
where the coefficients $c_m(t,\h)$ are all $\cO(\h\,(1+|t|^l))$ and the remainder $\|R_l\|_{L^2}\leq C\,(|t|\h)^{l+1}$ for $C>0$.

The square norm of $\Phi_{\chi_T,E_\h}$ is expressed by
\bequ\label{e:norm}
\|\Phi_{\chi_T,E_\h}\|^2 = \iint \la \phi^{(N)}_{t'},\phi^{(N)}_t\ra\ e^{i(t-t')(E_\h-q^0)/\h}\,\chi_T(t')\,\chi_T(t)\,dt\,dt'\,.
\eequ
To compute the square norm, we thus need to estimate the overlaps
$$
\la \cD_{t'q^1}\varphi_{2m'},\cD_{tq^1}\varphi_{2m}\ra = \la \varphi_{2m'},\cD_{(t-t')q^1}\varphi_{2m}\ra\,.
$$
The first case $m=m'=0$ is easy to compute (it is a simple Gaussian integral), and gives for any $\beta\in\IR$ \cite[Eq.40]{FND03}:
\bequ\label{e:phi0Dphi0}
\la \varphi_{0},\cD_{\beta}\varphi_{0}\ra = \frac{1}{\sqrt{\cosh(\beta)}}\,.
\eequ
This overlap decreases fast when $|\beta|\gg 1$. We now show that the other terms have a similar behaviour.
Since $\cD_{\beta}=e^{-i\beta\Oph^w(x\xi)/\h}$, by differentiating with respect to $\beta$ we get
$$
\partial_\beta \cD_{\beta} = -\frac{i}{\h}\Oph(x\xi)\cD_{\beta}=((a^*)^2-a^2)\cD_{\beta}\,,
$$
using the notations of \S\ref{propagsec} for raising and lowering operators. 
Since $(a^*)^2\varphi_0 = \sqrt{2}\varphi_2$,
we get
$$
\partial_\beta \la \varphi_{0},\cD_{\beta}\varphi_{0}\ra = \sqrt{2}\la \varphi_{0},\cD_{\beta}\varphi_{2}\ra\,,
$$
so that 
$$
\la \varphi_{0},\cD_{\beta}\varphi_{2}\ra= -\frac{\sinh(\beta)}{(2\cosh(\alpha))^{3/2}}\,.
$$
Differentiating once more, we obtain similar expressions for $\la \varphi_{0},\cD_{\alpha}\varphi_{4}\ra$ and $\la \varphi_{2},\cD_{\alpha}\varphi_{2}\ra$. By induction, we can obtain in this way explicit expressions for all overlaps 
$\la \varphi_{2m'},\cD_{\alpha}\varphi_{2m}\ra$, all of which have the form of linear combinations of derivatives of $(\cosh\alpha)^{-1/2}$. As a result, all these overlaps will decay like $e^{-|\alpha|/2}$ when $|\alpha|\to\infty$. In particular, for any $m,m'\geq 0$ there exists $C_{m,m'}>0$ such that
\bequ\label{e:mm'vs00}
\Big|\frac{\la \varphi_{2m'},\cD_{\alpha}\varphi_{2m}\ra}{\la \varphi_{0},\cD_{\alpha}\varphi_{0}\ra}\Big|\leq C_{m,m'}\,,\quad\text{uniformly for }\alpha\in\IR\,.
\eequ
(these overlaps are independent of $\h$).

Expanding the states $\phi^{(N)}_t$ as in \eqref{e:norm}, the first term $(m,m')=(0,0)$ takes the form
\bequ
I_{0,0}=\int \la\varphi_0,\cD_{(t-t')q^1}\varphi_0\ra\, e^{i(t-t')\theta_\h}\,\chi_T(t')\,\chi_T(t)\, dt\,dt'\,,\qquad \text{where we set }\theta_\h \defeq (E_\h-q^0)/\h.
\eequ
From the expression \eqref{e:phi0Dphi0} we see that the integrand is exponentially localized near the diagonal. This motivates us to operate the change of variables
$\tt=\frac{t+t'}{2}$, $r=t-t'$, to get
$$
I_{0,0}=\int \frac{e^{ir\theta_\h}}{\sqrt{\cosh(q^1r)}}\,\chi_T(\tt+r/2)\,\chi_T(\tt-r/2)\,d\tt\,dr\,.
$$
We are interested in large times $T\gg 1$, so it makes sense to Taylor expand the functions $\chi_T$ around the central value $\tt$. We use the 1st order Taylor expansion with intermediate value:
\bequ
\chi_T(\tt\pm r/2)=\chi_T(\tt)\pm\frac{r}{2}\chi'_T(\tt\pm r_\pm(\tt,r)/2)\quad\text{for some }r_\pm(\tt,r)\in(0,r)\,.
\eequ
The product of the two expansions is split into three terms:
\bequ\label{e:chichi}
\chi_T(\tt)^2 +  \frac{r}{2} \chi_T(\tt)[ \chi'_T(\tt+r_+/2) - \chi'_T(\tt- r_-/2) ]
-\frac{r^2}{4} \chi'_T(\tt+r_+/2)\chi'_T(\tt -  r_-/2)\,.
\eequ
Let us keep only the first term in the integral $I_{0,0}$, producing
$$
I_{0,0}^1=\int_{\IR}  \frac{e^{ir\theta_\h}}{\sqrt{\cosh(q^1r)}}\,dr\int_{\IR}\chi_T(\tt)^2 \,d\tt\defeq 
S_1(q^1,\theta_\h)\,T\,\|\chi\|_{L^2}^2\,.
$$
The function $S_1(q^1,\theta_\h)$ admits an explicit expression \cite[Eq.(60)]{FND03}:
\bequ\label{e:S_1}
S_1(q^1,\theta_\h) = \frac{1}{q^1\sqrt{2\pi}}\big|\Gamma\big(\frac{1}{4}+i\frac{\theta_\h}{2q^1} \big) \big|^2\,.
\eequ
This function is positive for all values of $q^1,\theta_\h$. Given $q^1$ it takes its maximum at $\theta_\h=0$ with value $S_1(q^1,0)\approx 5.244/q^1$, and decays exponentially when $\theta_\h\to\infty$. In our situation $q^1=\llambda+\cO(\h)$, and we have $\theta_\h=\cO(1)$, so this function is uniformly bounded from below and from above by positive constants.

The second and third terms in \eqref{e:chichi} are both supported in $[-T,T]$ and bounded above respectively by $\frac{|r|}{T} \|\chi\|_{L^\infty}\|\chi'\|_{L^\infty}$ and $\frac{r^2}{4T^2}\|\chi'\|_{L^\infty}^2$. Injected into the integral $I_{0,0}$, these terms produce integrals $I_{0,0}^2$, $I_{0,0}^3$ with the following bounds:
\begin{align*}
|I_{0,0}^2|&\leq \frac{1}{T}\int  \frac{|r|}{\sqrt{\cosh(q^1r)}}\,dr \int_{-T}^{T} d\tt\, \|\chi\|_{L^\infty}\|\chi'\|_{L^\infty} =S_2 \, \|\chi\|_{L^\infty}\|\chi'\|_{L^\infty}\,,\\
|I_{0,0}^3|&\leq \frac{1}{4T^2}\int \frac{r^2}{\sqrt{\cosh(q^1r)}}\,dr \int_{-T}^{T} \, d\tt\, \|\chi\|_{L^\infty}\|\chi'\|_{L^\infty} = \frac{S_3}{T}\,\|\chi'\|^2_{L^\infty}\,,
\end{align*}
where $S_2,S_3>0$ only depends on $q_1$. When $T\gg 1$, these two terms are subdominant compared with $I_{0,0}^1$. Taking into account that $q^1=\llambda+\cO(\h)$, we get
$$
I_{0,0}=T\,S_1(\llambda,\theta_\h)\,\|\chi\|_{L^2}^2\, (1 + \cO(1/T)).
$$
Let us now consider the parts of $\|\psi^{(N)}_t\|^2$ involving the corrective terms $c_{m}(t,\h) \cD_{tq^1}\varphi_{2m}$. 
Such a corrective term may be coupled to the main term $\cD_{t'q^1}\varphi_{0}$, or to another corrective term 
$c_{m'}(t',\h) \cD_{t'q^1}\varphi_{2m'}$, $0\leq m'\leq 2l$. In both cases, \eqref{e:mm'vs00} shows that the involved scalar product $\la\varphi_{2m'},\cD_{(t-t')q^1}\varphi_{2m}\ra$ is bounded above by $C_{m'm}\la\varphi_{0},\cD_{(t-t')q^1}\varphi_{0}\ra$. 
On the domain of integration, all the polynomials $c_{m}(t,\h)$ are uniformly bounded above by $\cO(\h T^l)$. These terms are thus bounded above by $\tilde{C}_{m,m'}\,\h\, T^l\,I_{0,0}(\theta_\h=0)$, and are thus much smaller than $I_{0,0}$.

There remains to treat the contribution to $\|\psi^{(N)}_t\|^2$ of the remainder $R_l$. Since all $\varphi_m$ are normalized, a simple computation shows that the parts of $\|\psi^{(N)}_t\|^2$ involving $R_l$ are globally bounded above by $C\,\h^{l+1}T^{l+3}$.

Summing up all contributions, we obtain the stated estimate for the norm of $\Phi_{\chi_T,E_\h}$, and hence the norm of $\Psi_{\chi_T,E_\h}$.
\end{proof}
Estimating the norm of $\Psi_{\chi_T,E_\h}$ was a lengthy, yet necessary step. Using the fact that the norm is not 
very small, it allows us to prove that $\Psi_{\chi_T,E_\h}$ is microlocalized near $\gamma_0$.
Let us define the normalized quasimode
$$
\tPsi_{\chi_T,E_\h}\defeq \frac{\Psi_{\chi_T,E_\h}}{\| \Psi_{\chi_T,E_\h} \|}\,.
$$
\begin{cor}\label{c:localization2}
If we choose the time $T = T_{\vareps'}$, the normalized state $\tPsi_{\chi_T,E_\h}$ is localized in the $\h^{\vareps'/3}$ neighbourood of $\gamma_0$, in the sense of Prop.~\ref{p:localization1}: for any $\Theta\in C^\infty_c(T^*(\IT\times \IR))$ with $\Theta\equiv 1$ in a fixed neighbourhood of $\gamma_0$, we have the estimate
\bequ
\|[\Oph(\Theta_{\h^{\vareps'/3}})-I] \tPsi_{\chi_T,E_\h} \|_{L^2} = \cO(\h^\infty)\,.
\eequ
\end{cor}
\begin{proof}
The proof is a direct consequence of Prop.~\ref{p:localization1}. Since $\Psi_{\chi_T,E_\h}$ is a linear combination of the evolved states $\psi_t^{(N)}$, we have by the triangular inequality, for any power $M>0$:
$$
\|[\Oph(\Theta_{\h^{\vareps'/3}})-I] \Psi_{\chi_T,E_\h}\|_{L^2}\leq T_{\vareps'}\|\chi\|_{L^\infty} \cO(\h^{M})\,.
$$
Since the norm of $\Psi_{\chi_T,E_\h}$ is of order $|\log\h|^{1/2}$, we may divide by it on both sides, keeping a remainder of the same form.
\end{proof}
We have shown in Lemma~\ref{l:E-local1} that the states composing $\Psi_{\chi_T,E_\h}$ are quasimodes for the operator $P^{(N)}(\h)$ which are centered at the energy $q^0(\h)=E_0+\cO(\h)$. 
We will take for central energy of our quasimode $\tPsi_{\chi_T,E_\h}$ the value $E_\h$ used in its definition, which may differ from $q^0$ by $\cO(\h)$. To compute the width of this quasimode we need to estimate the 
norm of $(P^{(N)}(\h)-E_\h)\Psi_{\chi_T,E_\h}$.
\begin{prop}\label{p:width-model}
For $T=T_{\vareps'}$, we have the norm estimate
$$
\|(P^{(N)}-E_\h)\Psi_{\chi_T,E_\h}\|^{2} = \frac{\h^2}{T_{\vareps'}}\,S_1(\llambda,(E_\h-q^0)/\h)\,\|\chi'\|_{L^2}^2\, (1 + \cO(1/|\log\h|))\,,
$$
where the function $S^1$ is defined in \eqref{e:S_1}.
\end{prop}
\begin{proof}
The proof is a straightforward adaptation of the proof of Lemma~\ref{l:norm-qmode}. From the definition (\ref{eq:def-qmode}) of $\Psi_{\chi_T,E_\h}$, we see that
\begin{align*}
(P^{(N)}-E_\h)\Psi_{\chi_T,E_\h}&= \int_{\IR}\chi_T(t)\,(P^{(N)}-E_\h)\,e^{-it(P^{(N)}-E_\h)/\h}\psi_0\,dt\\
&= \int_{\IR}\chi_T(t)\,e^{itE_\h/\h}\,i\h\partial_t(\,e^{-it(P^{(N)}-E_\h)/\h})\psi_0\,dt\\
&= -i\h \int_{\IR}(\partial_t\chi_T(t))\,e^{-it(P^{(N)}-E_\h)/\h}\psi_0\,dt\,.
\end{align*}
The integral in the RHS is similar to the one defining $\Psi_{\chi_T,E_\h}$. To compute its norm, we may just apply the Lemma~\ref{l:norm-qmode}, up to replacing $\chi$ by $\chi'$. We get:
$$
\|(P^{(N)}-E_\h)\Psi_{\chi_T,E_\h}\|^2 = \frac{\h^2}{T}\,S_1(\llambda,(E_\h-q^0)/\h)\,\|\chi'\|_{L^2}^2\, (1 + \cO(1/T))\,,
$$
which is the announced result after specializing to $T=T_{\vareps'}$.
\end{proof}
Putting together the results of Lemma~\ref{l:norm-qmode} and Prop.~\ref{p:width-model}, we find that the normalized quasimodes $\tPsi_{\chi_T,E_\h}$, centered at the energy $E_\h$, have the width 
$$
f(\h)=\frac{\h}{T_{\vareps'}}\frac{\|\chi'\|_{L^2}}{\|\chi\|_{L^2}}(1 + \cO(1/|\log\h|))\,,
$$
which is of the announced order $\cO(\h/|\log\h|)$. Let us pay attention to the factor. The time $T_{\vareps'}$ is the Ehrenfest time \eqref{e:Ehrenfest}. As in \cite{VS05}, we can optimize over the choice of cutoff $\chi$ to minimize the ratio $\frac{\|\chi'\|_{L^2}}{\|\chi\|_{L^2}}$.
\begin{lem}
\bequ\label{e:minimiz}
\inf_{\chi}\frac{\|\chi'\|_{L^2}}{\|\chi\|_{L^2}}=\pi/2\,,
\eequ
where we take the infimum over all $\chi\in C^\infty_c((-1,1),[0,1])$. Hence, 
for any $\vareps'>0$, we may find a cutoff $\chi_{\vareps'}\in C^\infty_c((-1,1),[0,1])$ such that 
$$
\frac{\|\chi_{\vareps'}'\|_{L^2}}{\|\chi_{\vareps'}\|_{L^2}}\leq \pi/2(1+\vareps')\,.
$$
\end{lem}
\begin{proof}
By relaxing the condition on the range of $\chi$, we may minimize the quadratic form 
$q(\chi)=\int(\chi'(t))^2\,dt$ over $\chi\in C^\infty_c((-1,1))$. This quadratic form can be used to define the Laplacian on the interval $(-1,1)$, with Dirichlet boundary conditions: this operator has the domain $H^1_0((-1,1))$, and admits $C^\infty_c((-1,1))$ as a core. Its ground state $\chi_0(t)=\cos(\pi t/2)\bbbone_{[-1,1]}$ reaches the announced infimum. By slightly perturbing $\chi_0$, we easily construct a smooth $\chi_{\vareps'}$ with the required property.
\end{proof}
Taking the cutoff $\chi_{\vareps'}$ in the formula for $\Psi_{\chi_T,E_\h}$, we get the width
\bequ\label{e:opt-width}
f(\h)=\pi\llambda (1+2\vareps')\frac{\h}{|\log\h|}\,(1 + \cO(1/|\log\h|))\,.
\eequ
Provided $\hbar$ is small enough, this width is bouded above by $C_\gamma\frac{\h}{|\log\h|}$ if we take $C_\gamma = \pi\lambda(1+3\vareps')$.

\subsection{A localized quasimode on $M$: proof of Proposition~\ref{pro:quasim0}}
To finish the proof of Proposition~\ref{pro:quasim0}, there remains to transform the quasimode $\tPsi_{\chi_T,E_\h}\in \cS_{\varphi/\h}(S^1 \times \R)$ for the QNF $P^{(N)}$ into a state $\psi_\h\in L^2(M)$, and check that the latter is a good quasimode for the operator $P(\h)$. This will be accomplished through the use of the Fourier integral operators $U_N$ bringing the original Hamiltonian $P(\h)$ to the QNF $P^{(N)}$, as described in Prop.~\ref{prop:qbnf}.
Using the notation of this Proposition, we consider the normalized quasimode $\tPsi_{\chi_T,E_\h}$ with the  optimized parameters $T_{\vareps'},\chi_{\vareps'}$ as described in the last subsection, and define the following states on $L^2(M)$, which will be quasimodes for the original operator $P(\h)$:
$$
\psi_\h \defeq U_N \,\tPsi_{\chi_T,E_\h}\,.
$$
Here we begin to manipulate our parameters. From Corollary~\ref{c:localization2} we know that $\tPsi_{\chi_T,E_\h}$ is, up to a remainder $\cO(\h^\infty)$, localized in the $\h^{\vareps'}/2$ neighbourood of $\gamma_0$. In this region, the FIO $U_N$ is essentially unitary, so that 
$$
\| \psi_\h \|_{L^2(M)} = 1 + \cO(\h^\infty)\,.
$$
Besides, conjugating the localization estimate of the Corollary, we
see that $\psi_\h$ is 
microlocalized in the $\h^{\vareps'}/2$ neighbourood of $\gamma(E_0)$, up to a remainder $\cO(\h^\infty)$.

Now, we want to study the energy width of the quasimode $\psi_\h$ with respect to the operator $P(\h)$. From the conjugation \eqref{e:QNF-conjugation}, we get
$$
U_N^{*}\, \,(P-E_\h)\, \psi_\h = (P^{(N)}-E_\h)\tPsi_{\chi_T,E_\h} + R_{N+1}\tPsi_{\chi_T,E_\h}\,.
$$
The microlocalization of $\psi_\h$ implies the similar localization of $(P-E_\h)\, \psi_\h$, so that 
$$
\|(P-E_\h)\, \psi_\h \|_{L^2(M)} = \|U_N^{*}\, \,(P-E_\h)\, \psi_\h  \|_{L^2(\IT\times\IR)} + \cO(\h^\infty)\,.
$$
(the FIO $U_N^*$ is unitary microlocally near $\gamma_0$, and subunitary away from it). 
We already know the norm of $(P^{(N)}-E_\h)\tPsi_{\chi_T,E_\h}$, given by the optimized width \eqref{e:opt-width}.
There remains to check that the norm of $R_{N+1}\tPsi_{\chi_T,E_\h}$ is of smaller order. 

We recall that the symbol of the operator $R_{N+1}$ has the form $r_{N+1}(s,\tau, x, \xi; h) = \cO\big((\h,x,\xi)^{N+1}\big)$ when $\h,x,\xi\to 0$. On the other hand, Prop.~\ref{c:localization2} shows that 
$\tPsi_{\chi_T,E_\h}$ is localized in a microscopic neighbourhood of $\gamma_0$, so we can write
$$
R_{N+1}\tPsi_{\chi_T,E_\h}=R_{N+1}\Oph(\Theta_{\h^{\vareps'/3}})\tPsi_{\chi_T,E_\h} + \cO(\h^\infty)\,.
$$
The symbol calculus in the class $S^{-\infty}_{\vareps'/3}(T^*(\IT\times\IR))$ shows that 
$$
R_{N+1}\Oph(\Theta_{\h^{\vareps'/3}})=\Oph(r_{N+1}\#\Theta_{\h^{\vareps'/3}})
$$
where $\#$ denotes the exact Moyal product on $T^*(\IT\times\IR)$.
The symbol $r_{N+1}\#\Theta_{\h^{\vareps'/3}}$ is of order $\cO(\h^{\vareps'(N+1)/3})$ in the class $S^{-\infty}_{\vareps'/3}(T^*(\IT\times\IR))$. The Calder\'on-Vaillancourt in this class then implies that
$$
\| R_{N+1}\Oph(\Theta_{\h^{\vareps'/3}}) \|_{L^2\to L^2}=\cO(\h^{\vareps'(N+1)/3})\,.
$$
As a consequence, we obtain
$$
\|(P-E_\h)\, \psi_\h \|_{L^2(M)} = \|(P^{(N)}-E_\h)\tPsi_{\chi_T,E_\h} \|_{L^2(\IT\times \IR)} + \cO(\h^{\vareps'(N+1)/3})\,.
$$
The first term on the RHS has been estimated in \eqref{e:opt-width}. To ensure that the remainder is of smaller order, we need $N,l$ and $\vareps'$ to satisfy the condition 
\bequ\label{e:Neps'}
(N+1)\vareps'/3>1\,,\qquad\text{and}\quad l\geq 2\,.
\eequ
\begin{rem}The choice of these parameters should proceed as follows. First we select $\vareps'>0$ small, which determines the choice of the time $T_{\vareps'}$ and the near-optimal cutoff $\chi_{\vareps'}$ used to construct our model quasimode. Then, we select $N$, the order of the normal form, large enough to satisfy the condition \eqref{e:Neps'}: this ensures that the remainder operator $R_{N+1}$ in the QNF, acting on our quasimode $\tPsi_{\chi_T,E_\h}$, gives a negligible contribution. 
The auxiliary index $l$ only determines the precision of our expansion for the evolved coherent states, but does not impact the definition of the quasimode.
\end{rem}
To end the proof of Proposition~\ref{pro:quasim0} there remains to check that 
$\psi_\h$ is microlocalized on $\gamma(E_0)$. 
Using the FIO $U_N$, we may transport the cutoffs $\Oph(\Theta_{\h^{\vareps'/3}})$ appearing in Prop.~\ref{p:localization1} and Corollary~\ref{c:localization2}, onto the operator
$$
 U_N\,\Oph(\Theta_{\h^{\vareps'/3}})\,U_N^* = \Oph(\tTheta) + \cO(\h^\infty)\,.
$$
In each coordinate chart, the symbol $\tTheta$ is well-defined up to a remainder in $\h^\infty S^{-\infty}_{\vareps'/3}(M)$. This symbol is essentially supported in the $\h^{\vareps'/3}$-neighbourhood of $\gamma(E_0)$, equal to unity in a slightly smaller neighbourhood. The operator $\Oph(\tTheta)$ is, up to a negligible error, selfadjoint (because $U_N$ are microlocally unitary near $\gamma$). 
It satisfies
\bequ\label{e:localization3}
\|(\Oph(\tTheta)-I)\psi_\h \| = \cO(\h^\infty)\,,
\eequ
showing that $\psi_\h$ is microlocalized on $\gamma(E_0)$. 

The above microlocalization implies that any semiclassical measure $\mu_{sc}$ associated with the sequence $(\psi_\h)$ must be a probability measure supported on $\gamma(E_0)$. 
Because the width $f(\h)=o(\h)$, this measure must be invariant with respect to the Hamiltonian flow $\Phi^t_p$. As a result, we must have $\mu_{sc}=\delta_{\gamma(E_0)}$ which the unique invariant measure supported on the orbit $\gamma(E_0)$. This achieves the proof of Prop.~\ref{pro:quasim0}

$\hfill\square$

\subsection{Comparison with another quasimode construction}\label{s:CdVP}
In \cite{CdVP94} Colin de Verdi\`ere and Parisse describe $\cO(\h^\infty)$ quasimodes and eigenstates of a 1D Hamiltonian with a hyperbolic fixed point $(x,\xi)=(0,0)$. Near the fixed point the Hamiltonian can be put in the QNF $Q^{(N)}(\h)$. Their strategy to construct the quasimodes is different from the one we presented here. Their starting Ansatz are the distributions on $\cS'(\IR)$:
$$
\tvarphi_\theta(x)\defeq |x|^{-1/2+i\theta}\,,\quad\text{for any $\theta\in\IR$}\,.
$$
A direct computation shows that $\tvarphi_\theta$ is an eigendistribution of the quadratic operator $Q^{(N)}_q$:
$$
Q^{(N)}_q \tvarphi_\beta = \h q^1 \theta\, \tvarphi_\beta\,.
$$
From the decomposition \eqref{e:q^(N)} and Lemma~\ref{l:x-xi}, $\tvarphi_\theta$ is also an eigendistribution of the full operator $Q^{(N)}$, with eigenvalue $\h q^1 \theta+ \cO(\h^2)$. This distribution appears to be a good starting point for a quasimode construction. In \cite{CdVP94} the authors use the fact that the fixed point is homoclinic: outside a neighbourhood of the fixed point, the unstable manifold ($x$-axis) bends and becomes the stable manifold ($\xi$-axis). This way, they are able to connect together the stable and unstable Lagrangian branches of $\tvarphi_\beta$ using WKB theory, and produce $\cO(\h^\infty)$ quasimodes of $P(\h)$. 

In our present setting of a ``generic'' fixed point, we have no information on the continuation of the stable and unstable manifolds. The best we can do to the distribution $\tvarphi_\theta$ is to microlocalize it inside a compact neighbourhood of $(0,0)\in T^*\IR$. For this aim we may use a cutoff $\Theta\in C^\infty_c([-2,2]^2)$, $\Theta\equiv 1$ in $[-1,1]^2$. The state 
$$
\Phi_{\Theta,\theta}\defeq \Oph^w(\Theta)\varphi_\theta
$$ 
is now in $L^2(\IR)$, with square norm $\| \Phi_{\Theta,\theta} \|^2= 2|\log\h|+\cO(1)$, each of the four branches of the stable and unstable manifolds carrying one fourth of this weight \cite{NV97}. 
One can show that the normalized state $\tPhi_{\Theta,\theta}=\frac{\Oph^w(\theta)\varphi_\theta}{\|\Oph^w(\theta)\varphi_\theta\|}$ converges to the semiclassical measure $\delta_{(0,0)}$. This state is a quasimode of $Q^{(N)}$ of central energy $\h q^1 \theta$. Let us compute its width:
\begin{align*}
(Q^{(N)}-\h q^1 \theta)\Phi_{\Theta,\theta} &= [Q^{(N)},\Oph^w(\Theta)] \varphi_\theta + \Oph^w(\Theta)\,(Q^{(N)}-\h q^1 \theta)\varphi_\theta\\
&= - i\h \Oph\big(\{q^{(N)},\Theta\}+\cO(\h)\big)\,\varphi_\theta\,.
\end{align*}
Because $\Theta(x,\xi)=1$ near the origin, the symbol
$\{q^{(N)},\Theta\}$ vanishes near the origin and is supported in an annulus. As a result, the state
$[Q^{(N)},\Oph^w(\Theta)]\,\varphi_\theta$ is microlocalized along the
four branches, away from the origin, and has opertor norm $\cO(\h)$. Finally, we find that $\tPhi_\Theta$ is a quasimode of width $\cO(\frac{\h}{|\log\h |^{1/2}})$. This is less sharp than the width $\cO(\h/|\log \h|)$ we have obtained by time averaging. 

We now proceed towards understanding the connection between these two quasimode constructions. If we replace in the integral \eqref{e:Phi_chi} the smooth time cutoff $\chi$ by a sharp cutoff $\bbbone_{[-1,1]}$, we obtain a quasimode with an energy width $\cO(\frac{\h}{|\log\h |^{1/2}})$, which resembles the state $\tPhi_{\Theta,\theta}$ with the identification $\h q^1 \theta = E_\h - q^0$. 
On the opposite, if we try to represent the state $\Phi_{\chi_T,E_\h}$ in the form of a truncation $\Oph^w(\Theta)\varphi_\theta$, we obtain a cutoff function of the form
$\Theta(x,\xi) \approx \chi\Big( \frac{\log((x^2+\xi^2)/\h)}{(1-\vareps')|log \h|}\Big)$ outside the disk $\{(x^2+\xi^2)\leq \h\}$. The corresponding symbol $\Theta$ is quite singular when $(x^2+\xi^2)\gtrsim \h$ as it belongs to the ``critical'' symbol class $S^{-\infty}_{1/2}(T^*\IR)$. This singular symbol would be the price to pay if one wants to recover the small energy width $\cO(\h/|\log \h|)$ using this construction.

\section{Partially localized quasimodes and small logarithmic widths}

In this section we will prove our main Theorem~\ref{thm:mainthm}.

We will use the assumption of compactness for $M$, which ensures that the spectrum of $P(\h)$ is purely discrete. 
In Prop.~\ref{pro:quasim0} we have constructed, for any given sequence $(E_\h=E_0+\cO(\h))_{\h\to 0}$, a family $(\psi_\h)_{\h\to 0}$ of quasimodes centered at $E_\h$ and with corresponding widths $C_\gamma\frac{\h}{|\log\h|}$, for a constant $C_\gamma=\pi\llambda(1+3\vareps')$.
In this section we will make the choice $E_\h = E_0$ for all $\h$.

For any eigenvalue $E_j=E_j(\h) \in \Spec(P(h))$, let
$\Pi_{E_j} \psi_\h$
be the corresponding spectral projection of the quasimode $\psi_\h$ on the $E_j$-eigenspace.
For some $c_2>0$ to be chosen later, we consider the spectral interval $I = [E_0 - c_2 \frac{\h}{|\log \h|}, E_0 + c_2 \frac{\h}{|\log \h|}]$. The Pythagorean theorem ensures that
\bequ \label{quasicent}
\|(P(h) - E_0)\psi_\h\|^2 = \sum_{E_j \in I} |E_j - E_0|^2 \|\Pi_{E_j} \psi_h \|^2 +  \sum_{E_j \in I^{\complement}} |E_j - E_0|^2 \|\Pi_{E_j} \psi_h \|^2.
\eequ
Using the quasimode property and the form of the interval $I$, the second term on the RHS above satisfies the following bound, for $\h\in(0,\h_{C_\gamma})$:
\bequ
\Big(\frac{C_\gamma \h}{|\log \h|} \Big)^2 \geq \sum_{E_j \in I^{\complement}} |E_j - E_0|^2 \|\Pi_{E_j} \psi_h \|^2 \geq \Big(\frac{c_2 \h}{|\log \h|} \Big)^2 \|\Pi_{I^{\complement}} \psi_h \|^2\,.
\eequ
As a result, we get
\bequ
\|\Pi_{I^{\complement}} \psi_h \|^2 \leq \Big( \frac{C_\gamma}{c_2} \Big)^2\,.
\eequ 
To ensure that this upper bound is nontrivial we choose $c_2 > C_\gamma$. As a consequence, the projection of $\psi_\h$ inside $I$ satisfies
\bequ\label{e:lower-b}
\|\Pi_{I} \psi_h\|^2\geq 1-\Big( \frac{C_\gamma}{c_2} \Big)^2>0\,,
\eequ
in particular the interval $I$ contains at least one eigenvalue $E_j$ of $P(\h)$.

In order to create quasimodes of smaller width than $\psi_\h$, we will project $\psi_\h$ on ``short'' spectral intervals. Precisely, for any choice of $\eps>0$, we take $K>0$ large enough such that 
\bequ
\frac{c_2}{K} \leq \epsilon\,,\qquad \text{for instance by taking }K=[c_2/\eps]+1\,,
\eequ 
and partition the interval $I$ into $K$ disjoint subintervals $(I_k)_{k=1,\ldots,K}$ of widths $\frac{2c_2}{K}$ and centered at the energies $e_k = E_0+\frac{c_2\h}{|\log\h|}(-1 + \frac{2k-1}{K})$.
By Pythagoras we have 
$$
\|\Pi_{I} \psi_h\|^2 = \sum_{k=1}^K \|\Pi_{I_k} \psi_h\|^2\,.
$$
For each $\h$, let $k(\h)$ be the index (or one of the indices) for which the norm $\|\Pi_{I_k} \psi_h\|$ is maximal. From the lower bound \eqref{e:lower-b} this means that
\bequ
\|\Pi_{I_{k(\h)}} \psi_h \|^2 \geq \frac{1}{K} \Big(1 - \Big( \frac{C_\gamma}{c_2} \Big)^2\Big) \,.
\eequ
The normalized state
\bequ
\tpsi_\h = \frac{\Pi_{I_{k_0(\h)}} \psi_\h}{\|\Pi_{I_{k_0(\h)}} \psi_\h\|}
\eequ
is automatically a quasimode of $P(\h)$, centered at $E_\h=e_{k(\h)}$ and of width $\eps \frac{\h}{|\log \h|}$. 

Let us now study the localization properties of the states $\tpsi_\h$. 
The orthogonality of the eigenfunctions of $P(h)$ implies the identity
$\|\Pi_{I_{k_0(\h)}} \psi_\h \|^2 = \la \tpsi_\h, \psi_\h  \ra^2$. 
Let us insert the microlocal cutoff $\Oph(\tTheta)$ used in 
\eqref{e:localization3}:
\begin{align*}
\la \tpsi_\h, \psi_\h  \ra^2 &= \la \tpsi_\h, \Oph(\tTheta)\psi_\h  \ra^2 +\cO(\h^{l})\\
&= \la \Oph(\tTheta)^*\tpsi_\h, \psi_\h  \ra^2 +\cO(\h^{l})\\
&\leq \| \Oph(\tTheta)\tpsi_\h\|^2 +\cO(\h^{l})\,,
\end{align*}
where we used the fact that $\Oph(\tTheta)^*=\Oph(\tTheta)+\cO(\h^\infty)$.
From these inequalities we deduce the lower bound
\bequ
\| \Oph(\tTheta)\tpsi_\h\|^2 \geq \frac{1}{K} \Big(1 - \big( \frac{C_\gamma}{c_2} \big)^2\Big) - \cO(\h^{l})\,.
\eequ
This estimate shows that any semiclassical measure $\mu_{sc}$ associated with the sequence of quasimodes 
$(\tpsi_\h)$ contains a singular component $w_\gamma \delta_{\gamma(E_0)}$, with a weight 
\bequ\label{e:weight1}
w_\gamma \geq \frac{1}{K} \big(1 - \big( \frac{C_\gamma}{c_2} \big)^2\big)\,.
\eequ
In other words, the quasimodes $(\tpsi_\h)$ exhibit a ``strong scar'' on the orbit $\gamma(E_0)$. 

Finally, we may optimize the parameters in the following way: for a given (small) width $\eps>0$, we want to keep the weight $w_\gamma$ as large as possible. 
By playing a bit with the parameters $c_2$ and $K$, we end up with the following estimates.
\begin{lem}
For any $0<\eps < C_\gamma$, there is a choice of $K$ and $c_2$ maximizing the RHS of \eqref{e:weight1}. 

When $\eps\ll C_\gamma$, the optimal parameters satisfy the estimates 
$c_2= \sqrt{3} C_\gamma + \cO(\eps)$, $K= K=[c_2/\eps]+1$, 
leading to the following lower bound on the weight:
\bequ
w_\gamma\geq \frac{\eps}{C_\gamma}\frac{2}{3\sqrt{3}} + \cO((\eps/C_\gamma)^2)\,.
\eequ
\end{lem}
Since we can take $C_\gamma$ arbitrary close to $\pi\lambda_\gamma$, we get the estimate \eqref{e:weight} of our Theorem~\ref{thm:mainthm}. $\hfill\square$



\begin{thebibliography}{HHHHH}


\bibitem[An06]{An06} N. Anantharaman, \textit{Entropy and the localization of eigenfunctions}, Ann. of Math. \textbf{168} 435--475 (2008)
 
\bibitem[AN07]{AN07} N. Anantharaman and S. Nonnenmacher, \textit{Half-delocalization for the Laplacian on an Anosov manifold}, Ann. Inst. Fourier \textbf{57},  2465-2523 (2007).

\bibitem[BabLaz68]{BabLaz68} V. Babich and V. Lazutkin, \textit{Eigenfunctions concentrated near a closed geodesic}, Topics in Math. Phys \textbf{2}, 9-18 (1968).

\bibitem[Bar06]{Bar06} A.H. Barnett, {\it Asymptotic rate of quantum ergodicity in chaotic Euclidean billiards}, Comm.
Pure Appl. Math. {\bf 59}, 1457–88 (2006)



\bibitem[Br13]{Br13} S. Brooks, \textit{Logarithmic-scale quasimodes that do not equidistribute}, preprint (2013).

\bibitem[BL14]{BL14} Sh. Brooks and  E. Lindenstrauss, \textit{Joint quasimodes, positive entropy, and quantum unique ergodicity}, Invent. Math. {\bf 198}, 219-259 (2014)

\bibitem[BZ04]{BZ04} N. Burq and M. Zworski, \textit{Geometric control in the presence of a black box}, J. Amer. Math. Soc. \textbf{17}, 443-471  (2004)

\bibitem[Chr07]{Chr07} H. Christianson, \textit{Semiclassical non-concentration around hyperbolic orbits}, J. Funct. Anal.  \textbf{246}, 145-195 (2007); Corrigendum, J. Funct. Anal. \textbf{258}, 1060-1065 (2010)

\bibitem[Chr11]{Chr11} H. Christianson, \textit{Quantum Monodromy and Non-concentration Near a Closed Semi-hyperbolic Orbit}, Trans. Amer. Math. Soc. {\bf 363}, 3373-3438 (2011)

\bibitem[CdV77]{CdV77} Y. Colin de Verdi\`ere, \textit{Quasimodes des vari\'et\'es Riemannienes}, Invent. Math. \textbf{43}, 14-52 (1977).

\bibitem[CdV85]{CdV85} Y. Colin de Verdi\`ere, \textit{Ergodicit\'e et fonctions propres de laplacien}, Comm. Math. Phys. \textbf{102}, 497-502 (1985).

\bibitem[CdVP94]{CdVP94} Y. Colin de Verdi\'ere and B. Parisse, \textit{\'Equilibre instable en r\'egime semi-classique I: Concentration microlocale}, Comm. PDE \textbf{19}, 1535-1563 (1994).

\bibitem[CR97]{CR97}  M. Combescure and D. Robert, \textit{Semiclassical spreading of quantum wavepackets and applications near unstable fixed points of the classical flow}, Asymp. Analysis \textbf{14}, 377-404 (1997).

\bibitem[dPBB94]{dPBB94} G.G. de Polavierja, F.Borondo and R.M.Benito, {\it Scars in groups of eigenstates in a classically chaotic system}, Phys. Rev. Lett. {\bf 73} 1613-1616 (1994) 

\bibitem[DimSj99]{DimSj99} M. Dimassi and J. Sj\"ostrand, \textit{Spectral asymptotics in the semiclassical limit}, London Mathematical Society Lecture Notes Series \textbf{268}, Cambridge University Press (1999).

\bibitem[DG13]{DG13} S. Dyatlov and C. Guillarmou, \textit{Microlocal limits of plane waves and Eisenstein functions}, Ann. Sci. l'ENS \textbf{47}, 371-448 (2014).

\bibitem[Duis74]{Duis74} H. Duistermaat, \textit{Oscillatory integrals, Lagrange immersions and unfolding of singularities}, Comm. Pure. App. Math. \textbf{27}, 207-281 (1974).

\bibitem[FND03]{FND03} F. Faure, S. Nonnenmacher and S. De Bi\`evre, \textit{Scarred eigenstates for quantum cat maps of minimal periods}, Comm. Math. Phys. \textbf{239}, 449-492 (2003).

\bibitem[FN04]{FN04} F. Faure and S. Nonnenmacher, \textit{On the Maximal Scarring for Quantum Cat Map Eigenstates}, 
Commun. Math. Phys. \textbf{245}, 201-214, (2004)


\bibitem[GL93]{GL93} P. G\'erard and G. Leichtnam, \textit{Ergodic properties for eigenfunction of the Dirichlet problem}, Duke Math. J. \textbf{71}, 559-607 (1993).

\bibitem[GS87]{GS87} C. G\'erard and J. Sj\"ostrand, \textit{Resonances generated by a closed hyperbolic orbit}, Comm. Math. Phys., \textbf{108} 391-421 (1987).

\bibitem[Gui96]{Gui96} V. Guillemin, \textit{Wave trace invariants}, Duke Math. J. \textbf{83}, 257-352 (1996).

\bibitem[GuiPaul09]{GuiPaul09} V. Guillemin and T. Paul, \textit{Some remarks on semiclassical trace formulas and quantum normal forms}, Comm. Math. Phys. \textbf{294}, 1-19 (2010).

\bibitem[GuiWein76]{GuiWein76} V. Guillemin and A. Weinstein, \textit{Eigenvalues associated with a closed geodesic}, Bull. A.M.S. \textbf{82}, 92-94 (1976).

\bibitem[Hag80]{Hag80} G. Hagedorn, \textit{Semiclassical quantum mechanics I}, Comm. Math. Phys. \textbf{71}, 77-93 (1980).

\bibitem[Hj99]{HJ99}G. Hagedorn and A.Joye, \textit{Semiclassical dynamics
with exponentially small error estimates}, Commun. Math. Phys. \textbf{207}, 439-465 (1999)

\bibitem[Hass10]{Hass10} A. Hassell, \textit{Ergodic billiards that are not quantum unique ergodic}, Ann. Math. \textbf{171}, 605-618 (2010).

\bibitem[H84]{H84} E. Heller, \textit{Bound-state eigenfunctions of classical chaotic hamiltonian systems: scars of periodic orbits}, Phys. Rev. Lett. \textbf{53}, 1515-1518 (1984).



\bibitem[Ho85]{Ho85} L. H\"ormander, \textit{The analysis of linear partial differential operators III}, Grundlehren der Mathematischen Wissenschaften [Fundamental Principles of Mathematical Sciences] \textbf{256}, Springer-Verlag (1985).

\bibitem[Kap99]{Kap99} L. Kaplan, \textit{Scars in quantum chaotic wavefunctions}, Nonlinearity {\bf 12}, R1-R40 (1999)

\bibitem[KH99]{KH99} L. Kaplan and E. Heller, \textit{Measuring scars of periodic orbits}, Phys. Rev. \textbf{E 59}, 6609-6628 (1999).

\bibitem[K58]{K58} J. Keller, \textit{Corrected Bohr-Sommerfeld quantum conditions for nonseparable systems}, Ann. Phys. \textbf{4} 180-188 (1958).


\bibitem[Lin06]{Lin06} E. Lindenstrauss, \textit{Invariant measures and arithmetic quantum unique ergodicity}, Ann. Math. \textbf{163}, 165-219 (2006).

\bibitem[Mas72]{Mas72} V. Maslov, \textit{Theory of perturbations and asymptotic methods} (in Russian), Moscow State University Press (1972).

\bibitem[NV97]{NV97} S. Nonnenmacher and A. Voros, \textit{Eigenstate structures around a hyperbolic point}, J. Phys. {\bf A
30}, 295–315 (1997)

\bibitem[Ral76]{Ral76} J. Ralston, \textit{On the construction of quasimodes associated with stable periodic orbits}, Comm. Math. Phys. \textbf{51}, 219-242 (1976).


\bibitem[RS94]{RS94} Z. Rudnick and P. Sarnak, \textit{The behavior of eigenstates of arithmetic hyperbolic manifolds}, Comm. Math. Phys. \textbf{161}, 195-213 (1994).

\bibitem[Sch74]{Sch74} A. \u{S}chnirel'man, \textit{Ergodic properties of eigenfunctions}, Uspehi Mat. Nauk \textbf{29}, 181-182 (1974).

\bibitem[Sj02]{Sj02} J. Sj\"ostrand, \textit{Resonances associated to a closed hyperbolic trajectory in dimension 2}, Asym. Anal. \textbf{36}, 93-113 (2003).

\bibitem[To96]{To96} J. Toth, \textit{Eigenfunction localization in the quantized rigid body}, J. Diff. Geom. \textbf{43}, 844-858 (1996).

\bibitem[To99]{To99} J. Toth, \textit{On the quantum expected values of integrable metric forms}, J. Diff. Geom. \textbf{52}, 327-374 (1999).

\bibitem [VC00]{VC00}  E.Vergini and G.Carlo, {\it Semiclassical quantization with short periodic orbits}, J.Phys {\bf A 33}  4717-4724 (2000),

\bibitem [VC01]{VC01} E.Vergini and G.Carlo, {\it Semiclassical construction of resonances with hyperbolic structure: the scar function}, J.Phys {\bf A 34}  4525-4552 (2001)

\bibitem[VS05]{VS05} E.Vergini and D.Schneider, {\it Asymptotic behaviour of matrix elements between scar functions}, J.Phys {\bf A 38} 587-616 (2005) 

\bibitem[Vor76]{Vor76} A. Voros, \textit{Semiclassical approximations}, Ann. Inst. Henri Poincar\'e \textbf{24}, 31-90 (1976).

\bibitem[Wein74]{Wein74} A. Weinstein, \textit{On Maslov's quantization condition}, Symposium on Fourier integral operators - Nice, Springer-Verlag (1976).

\bibitem[Zel87]{Zel87} S. Zelditch, \textit{Uniform distribution of eigenfunctions on compact hyperbolic surfaces}, Duke Math. J. \textbf{55}, 919-941 (1987).

\bibitem[Zel98]{Zel98} S. Zelditch, \textit{Wave trace invariants around a non-degenerate closed geodesic}, Geom. Anal. Func. Anal. \textbf{8}, 179-217 (1998).



\bibitem[ZZ96]{ZZ96} S. Zelditch and M. Zworski, \textit{Ergodicity for eigenfunctions of ergodic billiards}, Comm. Math. Phys. \textbf{175}, 673-682 (1996).

\bibitem[Z12]{Z12} M. Zworksi, \textit{Semiclassical Analysis}, Graduate Studies in Mathematics \textbf{138}, Amer. Math. Soc. (2012).

\end{thebibliography}
\end{document}